\documentclass[11pt]{amsart}
\usepackage[utf8]{inputenc}
\usepackage{a4wide}
\usepackage{amsmath,mathrsfs,eucal}
\usepackage{xcolor}
\usepackage{bbm,stix}
\usepackage{enumitem}
\usepackage{graphics} 
\usepackage{epsfig} 
\usepackage{graphicx,subfig}                     
\usepackage{wrapfig}
\usepackage{textcomp}
\usepackage{tikz}
\usepackage{pgfplots}
\usepackage{dsfont}
\pgfplotsset{compat=1.16}

\usepackage{multicol}  
\usepackage{epstopdf}
\usepackage[colorlinks=true,allcolors=black]{hyperref}

\numberwithin{equation}{section}
\newtheorem{mainthm}{Theorem}
\newtheorem{thm}{Theorem}[section]
\newtheorem{cor}[thm]{Corollary}
\newtheorem{remark}[thm]{Remark}
\newtheorem{lemma}[thm]{Lemma}

\newtheorem{proposition}[thm]{Proposition}

\newtheorem{definition}[thm]{Definition}

\newenvironment{hypothesis}{

\enumerate
}{
\endenumerate
}

\newcommand{\grad}{{\textnormal{grad}\,}}
\renewcommand{\div}{{\textnormal{div}\,}}

\newcommand{\eps}{\varepsilon}

\newcommand{\V}{\mathscr{V}}
\newcommand{\W}{\mathscr{W}}
\newcommand{\F}{\mathscr{F}}
\newcommand{\M}{\mathcal{M}}
\renewcommand{\P}{\mathcal{P}}
\newcommand{\R}{\mathbb{R}}

\newcommand{\N}{\mathbb{N}}
\newcommand{\Lip}{\operatorname{Lip}}

\newcommand{\spt}{\operatorname{spt}}
\newcommand{\J}{\mathcal{J}}

\def\weak{\rightharpoonup}

\title[Optimal control problems of nonlocal interaction equations]{Optimal control problems of nonlocal interaction equations}
\author{Simone Fagioli, Alic Kaufmann, Emanuela Radici}

\address[Simone Fagioli]{\newline Dipartimento di Ingegneria e Scienze dell’Informazione e Matematica\newline Universit\`a degli Studi dell’Aquila, Via Vetoio 1, 67100 Coppito, L’Aquila, It.}
\email{simone.fagioli@univaq.it}
\address[Alic Kaufmann]{\newline Institute of Mathematics\newline
EPFL, CH-1015 Lausanne, Switzerland}
\email{alic.kaufmann@alumni.epfl.ch}
\address[Emanuela Radici]{\newline Institute of Mathematics\newline
EPFL, CH-1015 Lausanne, Switzerland}
\email{emanuela.radici@epfl.ch}
\date{}

\begin{document}

\maketitle
\begin{abstract}
In the present work we deal with the existence of solutions for optimal control problems associated to transport equations.  The behaviour of a population of individuals will be influenced by the presence of a population of control agents whose role is to lead the dynamics of the individuals towards a specific goal.  The dynamics of the population of individuals is described by a suitable nonlocal transport equation,  while the role of the population of agents is designed by the optimal control problem.  This model has been first studied in \cite{bongini2017optimal} for a class of continuous nonlocal potentials, while in the present project we consider the case of mildly singular potentials in a gradient flow formulation of the target transport equation. 
\end{abstract}

\section{Introduction}
The modelling of self-organising system has been intensively investigated in recent decades. The different mechanisms underling the phenomena were largely studied, see \cite{CaFoToVe,CucDon,CucSma,KeMiAuWa,ViCzBeCoSh,ViZa}, and most of these works concern the description of different interaction rules such as attraction, repulsion and alignment, common in particle physics \cite{dobrushin,morrey}, cell and population biology \cite{Parisi,Cametal,CoKrFrLe,cristiani,PaViGr,Perthame}, and social sciences \cite{aletti,colombo1,CrPiTo,during,Short,toscani_opinion}. On the other side,  the control problem  of self- organising systems, namely the possibility of modifying the behaviour of agents by directing them towards a fixed target, while maintaining their usual rules of interaction, has produced an increasing number of contributions in the literature, see \cite{AbABMaAl,AlBoCrKa,AlHePa,AlPaZa,BoFoJuSc} and references therein.

As a result of the above considerations we can think to describe the system with a discrete set of $N$ interacting agents, or particles, with positions $X_1(t),...,X_N(t)\in \R^d$ depending on time, and with given masses $n_1,...,n_N > 0$. In a classical dynamic framework the evolution in time of the particles can be described through the Cauchy problem on $\R^{dN}$
\begin{equation*}
    \dot{X}_j(t)  =- \sum_{k =1}^N n_k K(X_j (t ) - X_k (t ))+u(t), 
\end{equation*}
for $j=1,\ldots,N$, where $u$ is a control variable, to be chosen in a set of admissible control functions, minimiser of a proper cost functional $\J:=\J\left(u,X_1,\ldots,X_N\right)$ taking into account the desired behaviour of the particles as well as the cost of the control. A natural modelling choice is to consider the control variable as a family of $M$ control particles $Y_1,...,Y_M\in \R^d$ of masses $m_1,...,m_M > 0$, interacting with $X_1(t),...,X_N(t)$. Thus, the problem can be formulated as follows
\begin{equation}\label{eq:control_discrete}
\begin{aligned}
       \dot{X}_j(t)  &=- \sum_{k =1}^N n_k K(X_j (t ) - X_k (t ))- \sum_{h=1}^M m_h  H(X_j (t ) - Y_h (t )),\quad j=1,\ldots,N,\\
       \bar{Y}(t) &= \operatorname*{argmin}\limits_{\bar{U}} \J\left(\bar{U}(t),\bar{X}(t)\right), 
\end{aligned}
\end{equation}
where $\bar{X}=(X_1,...,X_N)$, $\bar{Y}=(Y_1,...,Y_M)$ and $\bar{U}$ ranges over a set of admissible control vectors. 

In dealing with the optimisation problem \eqref{eq:control_discrete} one can face in the so-called \emph{curse of dimensionality}, see \cite{Bellman}, i.e. the difficulty in solving the problem when the dimension of $(X,Y)$ becomes large.

This dimensionality problem can be bypassed by introducing an optimal control strategy independent on the number of agents but depending on their distributions. More precisely, if  $\rho$  represents the distribution of the population of particles $(X_1,...,X_N)$, and $\nu$ is the distribution corresponding to the particles $(Y_1,...,Y_N)$, assuming that the  total mass of the population is conserved, the evolution equation in \eqref{eq:control_discrete} can be replaced by its continuous counterpart, that is the transport equation
\begin{equation*}\label{first-continuity-equation}
\partial_t\rho(t,x)=-\div(\rho(t,x) v_{\nu}(t,x)),
\end{equation*}
where the velocity field $v_\nu$ will depend on the distribution of the population $\rho$ and the distribution of the control agents $\nu$ via non-local interaction kernels, 
\[
v_\nu(t,x)=K\ast\rho(t,x)+H\ast\nu(t,x)=\int K(t,x-y)d\rho(t,y)+\int H(t,x-y)d\nu(t,y).
\]
In order to drive/control the dynamics of $\rho$, we minimize a functional $\J(\nu,\rho)$ under the constraint that the transport equation is satisfied.  
The functional could take into account the desired behavior of $\rho$ but also the cost of the control agents. 
More precisely,  we deal with the optimal control problem 
\begin{equation}\label{control-equation}
\inf\,\,\J(\nu,\rho)\quad \text{s.t.}\quad \partial_t\rho(t,x)=-\div[(K\ast\rho(t,x)+H\ast\nu(t,x))\rho(t,x)].
\end{equation}
The rigorous passage from the agent based optimisation problem \eqref{eq:control_discrete} to the continuous problem \eqref{control-equation} can be performed by applying the mean field game approach introduced by Lasry and Lions \cite{LasLio}, see  \cite{BuPiToTs,BuPiToTsRo,ForSol} and references therein for a more deep treatment of the topic. We also mention \cite{HeStWa} where a Galerkin-type discretization was used, and \cite{HaTa} where a BBGKY hierarchy approach to the finite dimensional optimal control problems to infinite dimensional control problems limit was performed.

Similarly to the finite dimensional case, the optimisation problem \eqref{control-equation} can be applied in several context such as evacuation problems \cite{AbABMaAl,AlBoCrKa,AlHePa,CrPiTo}, alignment in swarming dynamics of animals or robots \cite{CaFoPiTr,HaLiGu,HaWa} and social sciences \cite{AlPaZa,WoCaBo}.  Note that,  if $K$ corresponds to the gradient of an interaction potential $W$,  then our argument applies to a class of functionals $W$, including Morse and Yukawa type potentials,  with a wide range of applications in models in biology and materials science, \cite{BoBrCoHa,BoLoPeAl,DaRu,Yuk}.

Typical form for the functional $\J$ in \eqref{control-equation} is \[
\J(\nu,\rho) = \int_0^T \int_\Omega C(\nu(t,x),\rho(t,x)) dx\, dt,
\]
where the cost function $C$ describes a
certain mutual interaction between the measures $\nu$ and $\rho$ \cite{CaFoPiTr}. 
Examples for this cost function can be distance function $C(\nu(t,x),\rho(t,x))=c(t)|x-x_0|^p \nu(t,x)$ used in the evacuation problems or functions of the second moment of  $\rho$ that allow to control the alignment  of the species, as well as the clustering towards a target opinion, see \cite{CaFoPiTr,WoCaBo} and references therein for more details. 

The combination of a nonlocal transport equation and an optimisation problem, as formulated in \eqref{control-equation}, has been first studied in \cite{bongini2017optimal}, considering Lipschitz continuous densities $\rho$ in the space of probability measures with finite first moment and regular potentials $K$ and $H$.  The main novelty of the present manuscript is that we can provide a positive answer to problem \eqref{control-equation} in case of self-interactions kernels $K$ showing jump type singularities, thus answering an open question posed in  \cite{bongini2017optimal}. In order to develop this improvement we will consider potentials $K,H$ of the form $K = -\nabla W$ and $H = -\nabla V$, thus the nonlocal transport equation in \eqref{control-equation} can be rephrased as the following equation
\begin{equation}\label{eq: TE gradient flow}
\partial_t \rho(t,x) =  \div \big[\rho(t,x) (\nabla W \ast \rho (t,x) + \nabla V \ast \nu (t,x)) \big].
\end{equation} 
Existence and uniqueness of weak type solutions to the above equation can be investigated using the so-called Jordan Kinderlehrer and Otto (JKO) scheme, in the spirit of \cite{AGS,carrillo2011global,jko} or, more precisely, in the semi-implicit version of the scheme introduced in \cite{di2013measure}.
Indeed, equation \eqref{eq: TE gradient flow} can be formally reformulated as a gradient flow of an associated energy functional defined on the Wasserstein space of probability measures endowed with  the Wasserstein distance $d_{W_2}$, see \eqref{wass_dist} below.  The main difference in this case is that the associated energy functional is not static. Indeed, it will depend on the variable $t$ because of the presence of the cross interaction with the distribution of the control agents $\nu(t)$.  
Let us briefly exploit the formal gradient flow formulation of \eqref{eq: TE gradient flow}.  For every fixed time $t$, consider the energy functional $\F_{\nu(t)}: \P_2(\R^d)\to\R$ defined as 
\[
\F_{\nu(t)}(\rho) : =  \frac12\int_{\mathbb{R}^d}W\ast\rho d\rho +\int_{\mathbb{R}^d}V\ast\nu(t)d\rho.
\]
Then, computing the spatial gradient of the first variation of $\F_{\nu(t)}$ w.r.t. the measure $\rho$, we have
\begin{equation}\label{gradient-of-first-var}
\nabla\frac{\delta \F_{\nu(t)}}{\delta\rho}(t,x)=\nabla W\ast\rho(t,x)+\nabla V\ast\nu(t,x).
\end{equation}
Combining \eqref{eq: TE gradient flow} with \eqref{gradient-of-first-var} we obtain
\begin{equation*}
\partial_t\rho(t) =\div\Big(\rho(t)\nabla\frac{\delta \F_{\nu(t)}}{\delta\rho}\Big). 
\end{equation*}
In order to rigorously formulate the above equation as a gradient flow in the Wasserstein space, one should be able to identify the r.h.s. as follows
$$\div\Big(\rho(t)\nabla\frac{\delta \F_{\nu(t)}}{\delta\rho}\Big)=-\grad_{d_{W_2}} \F_{\nu(t)}(\rho(t)),$$
where the gradient should be understood as a tangent vector in the space of probability measures endowed with a proper Riemannian structure induced by the Wasserstein distance, \cite{otto}.   

Under the formalism sketched above, equation \eqref{eq: TE gradient flow} can be formally equivalently regarded as 
\begin{equation}\label{gradient-flow}
\partial_t\rho=-\grad_{d_{W_2}} \F_{\nu(t)}(\rho).
\end{equation}
For equation in that form the JKO-scheme, see \eqref{eq:JKO} below for a precise definition, can be interpreted as a semi-implicit Euler scheme in time, where the time dependence through the control measure $\nu$ is left explicit. More precisely, considering a time step $\tau$ and a sequence of times $t_k$, $k=0,1,\ldots$, one can approximate \eqref{gradient-flow} as follows
\begin{equation*}
    \rho(t_{k+1})=\rho(t_k)-\tau\grad_{d_{W_2}} \F_{\nu(t_k)}(\rho(t_{k+1})).
\end{equation*}
The above nonlinear equation can be solved by a \textit{minimising movements} approach, that results to be the following minimisation problem
\[
 \min_{\rho }\left\{\frac{1}{2\tau}d_{W_2}^2(\rho,\rho(t_k))+\F_{\nu(t_k)}(\rho)\right\},
\]
where the minimum is attempted over the space of probability measures.

Given these considerations, our contribution to the above line of research can be summarise as follows.  In Theorem \ref{thm: TE gradient flows} we will construct weak measure solutions to equation \eqref{eq: TE gradient flow} under the assumptions of  jump type singularities for the gradient of $W$. The proof uses a semi-implicit version of the JKO-scheme that allows to extend the standard compactness argument and improve the usual regularity in time, provided Lipschitz continuity in time of $\nu$. We then investigate existence of solutions to the optimisation problem \eqref{control-equation} in Theorem \ref{thm: OC gradient flows}. The regularity in time obtained in Theorem \ref{thm: TE gradient flows} allows us to require lower semi-continuity of the cost functional $\J$ w.r.t narrow convergence and not weak-$*$ as in \cite{bongini2017optimal}. Moreover, we can further improve this regularity in time considering suitable convexity assumptions on the interaction potentials.

\emph{Structure of the paper}: In Section \ref{preliminaries} we introduce some notation, set the main assumptions on the interaction potentials $W$ and $V$ in \emph{\hyperlink{self target}{(Self)} and \hyperlink{cross target}{(Cross)}} respectively and we state the main results in Theorems \ref{thm: TE gradient flows} and \ref{thm: OC gradient flows}. We conclude the section collecting some preliminary results.  In Section \ref{existence sol continuity equation} we prove Theorem \ref{thm: TE gradient flows} concerning the existence of weak measure solutions to the transport equation \eqref{eq: TE gradient flow}, proving that these solutions satisfy a suitable Lipschitz regularity in time. Finally,  Section  \ref{optimal-control-section} is devoted to the proof of Theorem \ref{thm: OC gradient flows} on the existence of solutions to the optimisation problem \eqref{control-equation}.

\section{Preliminaries}\label{preliminaries}

\subsection{The Wasserstein distance}
In this section we collect the basic definition and known results about Wasserstein distances  and probability measures that will be useful for our analysis. 

We denote by $\M(\R^d)$ the set of all positive finite measures and  $\M_M(\R^d) \subset \M(\R^d)$ the subset of measures with total mass less than $M$, \textit{i.e.} $\mu\in \M(\R^d)$ such that $\mu(\R^d)\leq M$. We call $\M_M^R(\R^d)$ the set of positive measures with mass smaller or equal than $M$ and support contained in the closure of the ball $B(0,R) \subset \R^d$.

With $\P(\R^d)$ we denote the set of probability measures and, if $p \geq 1$, $\P_p(\Omega)$ is the subset of $\P(\R^d)$ containing probability measures with finite $p$-moments, namely 
\begin{equation*}
\int_{\R^d}|x|^p d\mu(x)<+\infty.
\end{equation*}

The pushforward of a measure $\mu\in\P(\R^{d_1})$ by a function $f:\R^{d_1}\to\R^{d_2}$ is defined by
\begin{equation*}
f_{\#}\mu(A)=\mu(f^{-1}(A))\text{ for every Borel set }A\subset\R^{d_2}.
\end{equation*}
If $\rho\in\P(\R^{d_1}\times\mathbb{R}^{d_2})$ and $\pi_1$, $\pi_2$ designate the canonical projections defined on $\R^{d_1}\times\R^{d_2}$,  then ${\pi_1}_{\#}$ and ${\pi_2}_{\#}$ are called the first and second marginal of $\rho$. Given $\mu\in\P(\R^{d_1})$ and $\nu\in\P(\R^{d_2})$ we denote by $\Gamma(\mu,\nu)$ the set of all couplings between $\mu$ and $\nu$, i.e.  all the  probability measures in $\P(\R^{d_1}\times\R^{d_2})$ whose first marginal is $\mu$ and second marginal is $\nu$.

Let $p\geq 1$ and $\mu, \nu \in \P_p(\R^d)$, then their $p$-Wasserstein distance is defined as 
\begin{equation}\label{wass_dist}
d_{W_p}(\mu,\nu):=\left(\inf\limits_{\gamma\in\Gamma(\mu,\nu)}\int_{\R^d\times\R^d}|x-y|^pd\gamma(x,y)\right)^{1/p}.
\end{equation}
In what follows we will denote by $\Gamma_o(\mu,\nu)$ the set of optimal couplings between $\mu, \nu$,  namely $\Gamma_o(\mu,\nu)$ will contain those elements of $\Gamma(\mu,\nu)$ for which the infimum in the definition of $p$-Wasserstein distance is attained. 
It is a standard result to prove that $\Gamma_o(\mu,\nu)$ is non-empty.

Given $r > 0$ and $\nu \in \P_p(\R^d)$, we will denote by 
\[ \overline{B}_{d_{W_p}}(\nu,r)=\left\{\mu\in\P_p(\R^d):d_{W_p}(\nu,\mu) \leq r\right\} \]
the closed $p$-Wasserstein ball centered at $\nu$ of radius $r$. We further introduce the following spaces
\begin{equation}\label{Lip rho}
    \Lip_{L,d_{W_p}}(0,T; \P_p)=  \left\{  \mu: [0,T] \to \P_p(\R^d) : d_{W_p}(\mu(t),\mu(s)) \leq L|t-s|, 
 \forall \, s , t \in [0,T] \right\}, 
\end{equation}
for some $L > 0$,  and 
\begin{equation}\label{Lip nu}
\Lip_{L',d_*}(0,T; \M_M^R)= \left\{ \mu: [0,T] \to \M_M^R(\R^d) :  d_*(\mu(t),\mu(s)) \leq L' |t-s|, \forall t,s \in [0,T] \right\},
\end{equation}
for some $L',M,R > 0$,  where  and $d_*$ is a distance metrising the weak-$*$ convergence of measures. 

\subsection{Assumptions and main results} 
In the present work we consider potentials $W,V$ satisfying the following properties
\begin{hypothesis}
\item[\hypertarget{self target}{(Self)}] $W \in C(\R^d) \cap C^1(\R^d \setminus \{0\})$ is an even globally Lipschitz kernel such that $W(0)=0$,  $$W(x) \leq C(1 + |x|^2)$$ for some $C>0$ and $W$ is $\lambda$-convex for some $\lambda\leq 0$,
\item[\hypertarget{cross target}{(Cross)}] $V \in C^1(\R^d)$ is a globally Lipschitz function bounded from below by some $V_0 \in \R$ 
\end{hypothesis}

Being $W$ only Lipschitz continuous at the origin,  the term $\nabla W \ast \rho (t,\cdot)$ is not  well defined unless we better specify it.  By calling $\partial^0 W$ the element of minimal norm in the subdifferential of $W$ at $x$,  then \emph{\hyperlink{self target}{(Self)}} assumption ensures that 
\[  \partial^0 W \ast \mu (x) = \int_{\{ y \neq x\}} \nabla W(x - y)  d\mu(y) \ \ \text{ and } \ \   \partial^0 W \ast \mu \in L^2(\mu) \ \ \text{ for any } \  \mu \in \P_2(\R^d). \]

Concerning the control functional $\J$ we assume that
\begin{hypothesis}
\item[\hypertarget{cost target}{(Contr)}] $\J : A \to \R \cup \{+\infty\}$ be a control functional that is bounded from below and lower semi-continuous with respect to the narrow convergence of measures. 
\end{hypothesis}

In order to state rigorously our definition of weak solution of \eqref{eq: TE gradient flow}, let us first introduce the following piece of notation.  We denote by 
\begin{equation*}\label{C rho}
C_{d_n}(0,T; \P_2) = \left\{   \mu: [0,T] \to \P_2(\R^d) : \mu  \text{ is continuous w.r.t the narrow convergence of measures} \right\},
\end{equation*}   
where $d_n$ is a distance metrising the narrow convergence of measures. 

\begin{definition}[Weak measure solutions]\label{def: weak solutions}
Let $T>0$ and $\rho_0 \in \P_2(\R^d)$.  We say that $\rho \in C_{d_n}(0,T; \P_2)$ is a \emph{weak measure solution} of \eqref{eq: TE gradient flow} with initial datum $\rho_0$ if $\rho(0,\cdot) = \rho_0$ in $\P_2(\R^d)$,  $\partial^0 W \ast \rho \in L^1((0,T); L^2(\rho(t)))$ and for every $\phi \in  C_c^{\infty}((0,T)\times\R^d)$ it holds
\[
	\int_0^T \int_{\R^d}\left(\frac{\partial\phi}{\partial t}(t,x)+   (\partial^0 W \ast \rho (t,x) + \nabla V \ast \nu (t,x))   \cdot \nabla\phi(t,x)\right)d\rho(t,x)=0.
\]
\end{definition}

We are now in position to present the main results of this work. 

\begin{mainthm}[Transport problem]\label{thm: TE gradient flows}
Let $T>0$, $\rho_0 \in \P_2(\R^d)$ and $\nu \in \Lip_{L',d_*}(0,T;\M_M^R)$ be fixed.  Let then $W,V$ be a self and a cross interaction potential satisfying \emph{\hyperlink{self target}{(Self)} and \hyperlink{cross target}{(Cross)}} respectively.  
Then there exists a weak measure solution $\rho$ of \eqref{eq: TE gradient flow} with initial datum $\rho_0$ in the sense of Definition \ref{def: weak solutions} in the space $\Lip_{L,d_{W_2}}(0,T;\P_2)$, for some $L = L(M, \Lip(V), \Lip (W)) > 0$.
\end{mainthm}

\begin{mainthm}[Optimal control problem]\label{thm: OC gradient flows}
Let $T>0$ and $\rho_0 \in \P_2(\R^d)$,  and assume that $W,V$ are interaction potentials satisfying \emph{\hyperlink{self target}{(Self)},  \hyperlink{cross target}{(Cross)}} respectively. Introduce the space
\[
 A:= Lip_{L',d_*}(0,T;\M_M^R) \, \times \,  Lip_{L,d_{W_2}}(0,T;\P_2)  \\
\]
and the control functional $\J$ be under assumption \emph{\hyperlink{cost target}{(Contr)}}. Then the problem
\begin{equation}\label{eq: OC gradient flow}
\inf_A \, \J(\nu,\rho) \;  \text{ s. t. $\rho$ is a weak measure solution of \eqref{eq: TE gradient flow} with $\nu$ and initial datum }  \rho(0) = \rho_0 
\end{equation}
admits a solution.

Moreover, if we further assume that $V$ is $\lambda'$-convex for some $\lambda' \leq 0$,  then the minimization in \eqref{eq: OC gradient flow} still admits in the bigger space  
\[
 A':= Lip_{L',d_*}(0,T;\M_M^R) \, \times \,  C_{d_n}(0,T; \P_2). 
\]
\end{mainthm}

\subsection{Technical preliminaries}
In this  section of the Preliminaries, we collect a couple of auxiliary technical results which will be useful in the proof of Theorems \ref{thm: TE gradient flows} and \ref{thm: OC gradient flows}. 

\begin{lemma}[\cite{Billingsley1999optimal}, Theorem 2.8] \label{narroConvProd}
Let $(\mu_k)_{k\ge 1},(\nu_k)_{k\ge 1}\subset \P(\mathbb{R}^d)$ such that $\mu_k\weak\mu\in \P(\mathbb{R}^d)$ and $\nu_k\weak\nu\in \P(\mathbb{R}^d)$. Then we also have narrow convergence of the product measure, i.e. $\mu_k\otimes\nu_k\weak\mu\otimes\nu$.
\end{lemma}

The next result relates the lower semi-continuity of a function to that of the corresponding integrand functional. 

\begin{proposition}\label{lscOfFunctional}
Let $f:\mathbb{R}^d\to\mathbb{R}\cup\{+\infty\}$ be a lower semi-continuous function, bounded from below. Then the functional $J : \P(\R^d) \to \R \cup \{+\infty\}$ defined as $J(\mu)=\int fd\mu$ is lower semi-continuous with respect to the narrow convergence of measures.
\end{proposition}
\begin{proof}
Thanks to Baire Theorem,  it is possible to find a non-decreasing sequence of continuous functions  $f_k : \R^d \to \R$ converging pointwise to $f$.  Up to a careful truncation argument,  it is always possible to assume that the functions $f_k$ are bounded from above (in general,  not uniformly in $k$).  On the other hand,  being $f$ bounded from below by assumption,  we can always assume that $f_k + c \geq 0$ for some $c \in \R$. 

By the monotone convergence Theorem we deduce that 
\[ \lim_{k \to \infty} \int_{\R^d} (f_k-c)d\mu =  \int_{\R^d} (f-c)d\mu  \quad \text{ for every } \mu \in \P(\R^d)  \]
which in turn implies that $\lim_{k \to \infty} J_k(\mu) = J(\mu)$, where we denoted $J_k(\mu)$ the functional $\int f_kd\mu$.  Since the functionals $J_k$ are non-decreasing in $k$ and they are all continuous with respect to the narrow convergence of measures,  i.e.  
\[  d_n(\mu_h, \mu) = 0 \ \ \Longrightarrow \ \ J_k(\mu_h) \to J_k( \mu),   \]
we infer that $J$ is lower semi-continuous with respect to the narrow convergence. 
\end{proof}

A straightforward consequence is the lower semi-continuity of the optimal transport cost.

\begin{proposition}[\cite{ambrosio2013user,AGS,santambrogio2015optimal}]\label{lscTranpCost}
Let $c:\mathbb{R}^d\times\mathbb{R}^d\to\mathbb{R}_{\ge0}$ be a non-negative lower semi-continuous cost function and $\mu,\nu \in \P(\R^d)$. Then the optimal transport cost 
\begin{equation*}
C(\mu,\nu)=\inf\limits_{\gamma\in\Gamma(\mu,\nu)}\int_{\mathbb{R}^d\times\mathbb{R}^d}c(x,y)d\gamma(x,y)
\end{equation*}
is lower semi-continuous with respect to narrow convergence of measures.
\end{proposition}

\begin{cor}\label{lscWp} 
Since the cost $c(x, y) = |x - y|^p$ satisfies the conditions of Proposition \ref{lscTranpCost},  the $p$-Wasserstein distance is lower semi-continuous with respect to the narrow convergence of measures.
\end{cor}

The next Lemma shows that closed $p$-Wasserstein balls are sequentially compact in $\P_p(\R^d)$ with respect to the $q$-Wasserstein topology for every $1 \leq q < p$.  Despite this result being well known in the literature, we report the proof for completeness.

\begin{proposition}\label{wassBallCompq12}
Let $\nu\in\P_p(\R^d)$ and $(\mu_k)_{k}\subset\P_p(\R^d)$ such that $(\mu_k)_{k} \subset \overline{B}_{d_{W_p}}(\nu,r)$ for some $r >0$.  Then there exists $\mu\in\P_p(\R^d) \cap \overline{B}_{d_{W_p}}(\nu,r)$ such that $W_q(\mu_k,\mu)\to0$ for all $q\in[1,p)$.
\end{proposition}

\begin{proof}
In order to prove the statement it is enough to show that the $q$-moments of $(\mu_k)_k$ are uniformly integrable for every $q \in [1,p)$.  Then tightness of the sequence $(\mu_k)_k$ and the convergence of the $q$-moments will follow as a consequence.  

Our first aim is to prove that for every $\eps > 0$ there exists some $R = R(\eps) >0$ such that 
\begin{equation}\label{eq:uniform integrability q moments}
\int_{\{|x|>R\}}|x|^q d\mu_k(x)<\eps \quad \text{ for all } \  k\in\N \ \text{ and } q \in [1,p).
\end{equation}
In order to do this, we first need a uniform estimate on the $p$-moments of the sequence $(\mu_k)_k$.  This comes straightforward by the $p$-Wasserstein bound,  indeed for each $k \in \N$ we can find $\gamma_k \in \Gamma_o(\mu_k,\nu)$ and compute 
\begin{equation}\label{eq:uniform bound p moments}
\int_{\R^d} |x|^p d \mu_k(x) = \int_{\R^d \times \R^d} |x|^p d \gamma_k(x,y) \leq 2^p d_{W_p}(\mu_k,\nu)^p + 2^p \int_{\R^d} |y|^p d\nu(y) \leq C < \infty
\end{equation}
for some $C = C(r,p, \nu) >0$ independent on $k$.  

Let now $q \in [1,p)$ and $\eps > 0$ be fixed,  then \eqref{eq:uniform integrability q moments} is a consequence of \eqref{eq:uniform bound p moments}.  Indeed
\begin{align*}
\int_{\{|x|>R \}}|x|^q  d\mu_k(x)= &\int_{\{|x|>R \}}\frac{1}{|x|^{p-q}}|x|^p d\mu_k(x)\\
\le&\frac{1}{R^{p- q}}\int_{\{|x|>R\}}|x|^p d\mu_k(x)\\
\le &\frac{1}{R^{p- q}}C < \eps \quad \text{ for all } k \in \N
\end{align*}
as soon as $R$ is big enough depending on $\eps,  p,q,C$. 

Let us observe that the above computation actually holds for every $q \in [0,p)$.  In particular, with the choice $q = 0$,  we deduce that the sequence $(\mu_k)_k$ is tight and thus, by Prokhorov's Theorem,  a not relabeled subsequence narrowly converges to some $\mu \in \P(\R^d)$.
Moreover,  applying the monotone convergence theorem together with \eqref{eq:uniform bound p moments} it is easy to see that  the limit measure $\mu$ has finite $p$-moment,  i.e.  it is an element of $\P_p(\R^d)$.
Moreover,  thanks to the lower semi-continuity of the $p$-Wasserstein distance recalled in Corollary \ref{lscWp},  we immediately deduce that $\mu \in \overline{B}_{d_{W_p}}(\nu,r)$.

To conlcude, we only need to prove that the $q$-moments of $\mu_k$ converge to the $q$-moment of $\mu$ for every $q \in [1,p)$.  

Thanks to \eqref{eq:uniform integrability q moments} and the fact that $\mu \in \P_p(\R^d)$, thus also in $\P_q(\R^d)$ for all $q \in [1,p)$,  for every $\eps>0$ it is possible to find some $R = R(\eps,q) > 0$ big enough so that 
\[ \int_{|x| > R} |x|^q \,d\mu_k(x) < \frac{\eps}{3}  \ \  \text{ for all $k$,  and } \ \ \left| \int_{\R^d} \big( |x|^q - \min\{ |x|^q , R \} \big) d\mu(x) \right| < \frac{\eps}{3}. \] 
Then,  since $\mu_k$ narrowly converges to $\mu$,  we can find some $\bar{k} = \bar{k}(\eps, R) \in \N$ such that for all $k \geq \bar{k}$ the following estimate holds 
\begin{align*}
  \left|  \int_{\R^d} |x|^q  d (\mu_k - \mu)(x)  \right|  \leq \,&  \left|  \int_{\R^d} \big(|x|^q -  \min\{ |x|^q , R \} \big) d (\mu_k - \mu)(x)  \right| \\
  & + \left|   \int_{\R^d}  \min\{ |x|^q , R \} d(\mu_k - \mu)(x)   \right| < \eps,  
 \end{align*}
thus concluding the proof. 
\end{proof}

\begin{remark}\label{wass-ball-tight and closed}
Let us emphasise two aspects in the proof of Proposition \ref{wassBallCompq12}.  The first one is that Wasserstein balls are always tight.  Secondly,  the closed $p$-Wasserstein ball is compact with respect to the $q$-Wasserstein topology for any $q \in [1,p)$ and with respect to the narrow convergence of measures. 
\end{remark}

\begin{lemma}\label{lem: piecewise interpolation}
Let $T > 0$ and $\nu \in \Lip_{L',d_*}(0,T;\M_M^R)$ for some fixed constants $L',M,R >0$.  For every $k \in \N$,  consider $\tau_k := T / k$ and the piecewise constant curves $\nu^k : [0,T] \to \M_M^R(\R^d)$ defined as
\[ \nu^k(t) := \sum\limits_{i=0}^{k-1} \nu(\tau_k(i+1)) \mathbbm{1}_{[\tau_k i,\tau_k(i+1))}(t).  \]
Then $d_n(\nu^k(t), \nu(t)) \to 0$ as $k \to \infty$ for every $t \in [0,T]$.
\end{lemma}

\begin{proof}
We observe that for every $t \in [0,T]$ the sequence $(\nu^k(t))_k$ is contained in $\M(\R^d)_M^R$,  which is a compact set in the metric space $(\mathcal{M}(\mathbb{R}^d),d_{\ast})$. Moreover,  given $0 \leq s < t \leq T$ we can consider for each $k \in \N$ the points
\[ s(k) = \min \{  \tau_k i \geq s : i \leq k  \} \quad \text{ and } \quad  t(k) = \min \{  \tau_k i \geq t : i \leq k  \},    \]
then $s(k)\leq t(k)$ and $s(k) \to  s, \,  t(k) \to  t$ as $k \to \infty$.

By construction,  it holds 
\[   d_* (\nu^k(s), \nu^k(t)) = d_* (\nu(s(k)), \nu(t(k))) \leq L' (t(k) - s(k))  \]
and passing to the limsup in $k$ we obtain 
\[  \limsup_{k \to \infty}  d_* (\nu^k(s), \nu^k(t)) \leq L' (t-s).  \] 

Thanks to the refined version of Ascoli-Arzel\`a Theorem (see \cite{AGS},  Proposition 3.3.1) we deduce that, up to a not relabeled subsequence,  $d_*(\nu^k(t),  \bar{\nu}(t)) \to 0$ for every $t \in [0,T]$ for some curve $\bar{\nu} : [0,T] \to \M_M^R(\R^d)$ that is $d_*$-continuous (i.e.  continuous with respect to the weak-$*$ topology).  It is immediate to observe that $\bar{\nu} \in Lip_{L',d_*} ([0,T],  \M_M^R(\R^d))$, indeed for every $0\leq s < t \leq T$ one has 
\[  d_*(\bar{\nu}(s), \bar{\nu}(t)) \leq  \limsup_{k\to \infty} \big[d_*(\bar{\nu}(s), \nu^k(s)) +  d_*(\nu^k(s), \nu^k(t)) +  d_*(\nu^k(t), \bar{\nu}(t)) \big] \leq L'(t-s).   \] 

We claim that $\bar{\nu} = \nu$.  Indeed,  let $t \in [0,T]$ and $\eps > 0$ be fixed.  Then there exists some $\bar{k} = \bar{k}(\eps) > 0$ big enough such that for every $k \geq \bar{k}$ one has 
\[  d_*( \nu^k(t), \bar{\nu}(t)) < \frac{\eps}{2} \quad \text{ and } \quad L' \tau_k  <  \frac{\eps}{2}.   \]
In particular, for such values of $k$ we can compute 
\[ d_*(\nu(t), \bar{\nu}(t)) \leq  d_*(\nu(t),  \nu(t(k))) +   d_*(\nu^k(t), \bar{\nu}(t)) \leq L'(t - t(k)) + \frac{\eps}{2}  < L' \tau_k + \frac{\eps}{2} < \eps,   \]
and passing to the limit as $\eps \to 0$ we can conclude that $d_*(\nu^k(t), \nu(t)) \to 0$ for every $t \in [0,T]$. 

Finally,  since $\spt(\nu^k(t)) \subseteq \overline{B(0,R)}$ for all $k\in\N$ and $t \in [0,T]$,  we deduce that $(\nu^k(t))_k$ is tight and hence, by Prokhorov's Theorem,  $d_n(\nu^k(t), \nu(t)) \to 0$ for every $t \in [0,T]$. 
\end{proof}

The rest of this section is devoted to prove uniqueness and stability properties of weak measure solutions of \eqref{eq: TE gradient flow} in the class of $2$-Wasserstein absolutely continuous curves,  under the further assumption that the cross interaction potential $V$ also enjoys the $\lambda$-convexity property.

\begin{lemma}\label{lem: lambda-conv-F}
Let $t \in [0,T]$,  $\nu \in \Lip_{L',  d_*}(0,T; \M_M^R)$,  $W,V$ be as in  \emph{\hyperlink{self target}{(Self)},  \hyperlink{cross target}{(Cross)}} respectively and assume that $V$ is $\lambda'$-convex for some $\lambda' \leq 0$. 
Consider the energy functional $\F_{\nu(t)}: \P_2(\mathbb{R}^d) \to\mathbb{R}$ defined as 
\[ \F_{\nu(t)} (\mu) :=  \frac12\int_{\mathbb{R}^d}W\ast\mu d\mu + \int_{\mathbb{R}^d}V\ast\nu(t)d\mu    \]
and,  given $\xi, \eta\in\P_2(\R^d)$ and $\gamma\in\Gamma_o(\xi,\eta)$,  consider the interpolating curve $\gamma_s :=((1-s)\pi_1+s\pi_2)_{\#}\gamma$ for $s \in [0,1]$.
Then $\F_{\nu(t)}$ enjoys the following convex inequality
\begin{equation}\label{eq: lambda-conv-F}
\F_{\nu(t)}(\gamma_s)\le (1-s) \F_{\nu(t)}(\xi)+s \F_{\nu(t)}(\eta)- \left(1+\frac{M}{2}\right)\bar \lambda s(1-s)d_{W_2}^2(\xi,\eta)
\end{equation}
where $\bar \lambda = \min\{ \lambda,  \lambda'\}$.
\end{lemma}

\begin{proof}
For simplicity let us write $ \F_{\nu(t)} = \W + \V_{\nu(t)}$,  where we have set 
\begin{equation}\label{def:functionals}
     \W(\mu) := \frac12\int_{\mathbb{R}^d}W\ast\mu \,d\mu \quad \text{ and } \quad  \V_{\nu(t)}(\mu) :=  \int_{\R^d}V\ast\nu(t) \, d\mu.  
\end{equation}    

If $\gamma$ and $\gamma_s$ are as in the statement,  thanks to the convexity properties of the potentials $W,V$, we can compute

\begin{align*}
\W(\gamma_s)=&\int_{\R^d \times \R^d} W(x-y)d\gamma_s(y)d\gamma_s(x)\\
=&\int_{\R^d \times \R^d \times \R^d \times \R^d}  W\underbrace{\left((1-s)x_1+sx_2-((1-s)y_1+sy_2)\right)}_{= (1-s)(x_1-y_1)+s(x_2-y_2)}d\gamma(y_1,y_2)d\gamma(x_1,x_2)\\
\le& \, (1-s)\int_{\R^d \times \R^d} W(x_1-y_1)d\xi(y_1)d\xi(x_1)+s\int_{\R^d \times \R^d} W(x_2-y_2)d\eta(y_2)d\eta(x_2)\\
&\, -\frac{\lambda}{2}(1-s)s\int_{\R^d \times \R^d \times \R^d \times \R^d} |x_1-y_1-(x_2-y_2)|^2d\gamma(x_1,x_2)d\gamma(y_1,y_2)\\
\le&\, (1-s)\W(\xi)+s\W(\eta)-\lambda(1-s)s d_{W_2}^2(\xi,\eta),
\end{align*}
and similarly
\begin{align*}
\V_{\nu(t)}(\gamma_s)=&\int_{\R^d \times \R^d} V(x-y)d\nu(t)(y)d\gamma_s(x)\\
=&\int_{\R^d \times \R^d} \int_{\R^d} V(\underbrace{(1-s)x_1+sx_2-y}_{(1-s)(x_1-y)+s(x_2-y)})d\nu(t)(y)d\gamma(x_1,x_2)\\
\le&\, (1-s)\int_{\R^d \times \R^d} V(x_1-y)d\nu(t)(y)d\xi(x_1)+s\int_{\R^d \times \R^d} V(x_2-y)d\nu(t)(y)d\eta(x_2)\\
&\, -\frac{\lambda'}{2}(1-s)s\int_{\R^d \times \R^d} |x_1-x_2|^2 \int_{\R^d} d\nu(t)(y)d\gamma(x_1,x_2)\\
\le&\,  (1-s)\V_{\nu(t)}(\xi)+s\V_{\nu(t)}(\eta)- \lambda' \frac{M}{2}(1-s)s  d_{W_2}^2(\xi,\eta).
\end{align*}
Then \eqref{eq: lambda-conv-F} follows by summing up the two above estimates and recalling that $\bar \lambda < \lambda, \lambda'$.
\end{proof}

\begin{proposition}[Stability for the Transport problem]\label{uniqueness of weak solutions}
Let $T>0$,  $\nu \in \Lip_{L',  d_*}(0,T;\M_M^R)$,  $W,V$ be as in  \emph{\hyperlink{self target}{(Self)},  \hyperlink{cross target}{(Cross)}} respectively and assume that $V$ is $\lambda'$-convex for some $\lambda' \leq 0$.
Given $\varrho_0,  \hat\varrho_0 \in \P_2(\R^2)$,  let $\varrho, \hat \varrho$ be two weak measure solutions of \eqref{eq: TE gradient flow} with initial datum $\varrho_0$ and $\hat\varrho_0$ respectively in the sense of Definition \ref{def: weak solutions}.  Moreover, assume that $\varrho, \hat \varrho$ are absolutely continuous curves with respect to the $2$-Wasserstein distance.  Then the following stability estimate holds
\begin{equation}\label{eq: stability estimate}
d_{W_2}(\varrho(t), \hat\varrho(t)) \leq  d_{W_2}(\varrho(0), \hat\varrho(0)) e^{-\bar\lambda(M+2)t} \quad \text{ for all } t \in [0,T],
\end{equation}
where $\bar \lambda = \min \{ \lambda,  \lambda' \}$.
In particular, weak measure solutions of \eqref{eq: TE gradient flow}  for the same initial datum are unique  in the class of $2$-Wasserstein absolutely continous curves.
\end{proposition}

\begin{proof}
The main tool of this proof consists in showing that any weak measure solution $\rho$ of \eqref{eq: TE gradient flow} that is absolutely continuous in $(\P_2(\R^2), d_{W_2})$ satisfies the following Evolution Variational Inequality
\begin{equation}\label{EVI}
\frac12\frac{d}{dt}d_{W_2}^2(\rho(t),\eta)+ \left(1 + \frac{M}{2} \right) \bar \lambda d_{W_2}^2(\rho(t),\eta)\le \F_{\nu(t)}(\eta)-\F_{\nu(t)}(\rho(t))
\end{equation}
for all $t \in [0,T]$ and $\eta \in \P_2(\R^d)$.
Indeed,  the claimed stability follows as an immediate consequence of \eqref{EVI} because,  choosing first 
$\rho(t)=\varrho(t)$ and $\eta=\hat \varrho(t)$ and then $\rho(t)=\hat \varrho(t)$ and $\eta=\varrho(t)$,  up to sum the two inequalities, we get 
\begin{equation*}
\frac{d}{dt}d_{W_2}^2(\varrho(t),\hat \varrho(t))+ (M + 2)\bar \lambda d_{W_2}^2(\varrho(t),\hat \varrho(t))\le 0.
\end{equation*}
Then a Gr\" onwall type argument implies 
\[d_{W_2}^2(\varrho(t),\hat\varrho(t))\le d_{W_2}^2(\varrho(0),\hat\varrho(0))e^{-\bar\lambda(M+2) t} \]
and, in turn,  the desired stability estimate \eqref{eq: stability estimate}.

Therefore, to conclude, we are left to prove the validity of inequality \eqref{EVI}. 
Thanks to \eqref{eq: lambda-conv-F} of Lemma \ref{lem: lambda-conv-F},  if $\gamma \in \Gamma_o(\xi,\eta)$ and $\gamma_s :=((1-s)\pi_1+s\pi_2)_{\#}\gamma$  for $s \in [0,1]$,  we know that 
\[  \frac{\F_{\nu(t)}(\gamma_s)-\F_{\nu(t)}(\xi)}{s}\le -\F_{\nu(t)}(\xi)+\F_{\nu(t)}(\eta)- \left( 1 + \frac{M}{2}\right) \bar \lambda
(1-s)d_{W_2}^2(\xi,\eta). \]

Choosing $\xi=\rho(t)$ for some $t \in [0,T]$ and passing to the liminf as $s \to 0$ we then obtain 
\begin{equation}\label{EVI-intermediate}
\liminf\limits_{s\searrow 0}\frac{\F_{\nu(t)}(\gamma_s)-\F_{\nu(t)}(\rho(t))}{s}+
\left( 1 + \frac{M}{2}\right) \bar\lambda  d_{W_2}^2(\rho(t),\eta) \le \F_{\nu(t)}(\eta)-\F_{\nu(t)}(\rho(t)).
\end{equation}

If we now call 
\begin{equation*}
v(t)(x)=\int_{x\neq y}\nabla W(x-y)d\rho(t)(y)+\nabla V\ast\nu(t)(x),
\end{equation*}
then Proposition 2.2 of \cite{carrillo2011global} implies that
\begin{equation}\label{directional-der-estimate}
\liminf\limits_{s\searrow 0} \frac{\F_{\nu(t)}(\gamma_s)-\F_{\nu(t)}(\rho(t))}{s}\ge \int_{\R^d \times \R^d}  v(t)(x) \cdot (y-x) d\gamma(x,y).
\end{equation}

Once here,  since $\rho$ is assumed to be absolutely continuous, we can apply Theorem 8.4.7 and Remark 8.4.8 of \cite{AGS} to infer that
\[ \int_{\R^d \times \R^d}  v(t)(x) \cdot (y-x) d\gamma(x,y) =  \frac12\frac{d}{dt}d_{W_2}^2(\rho(t),\eta).   \]

Finally,  \eqref{EVI} follows combining the above identity with \eqref{EVI-intermediate} and \eqref{directional-der-estimate}.
\end{proof}

\section{Transport problem}\label{gradient-flow-approach-section}\label{existence sol continuity equation}

In this section we present the proof of Theorem \ref{thm: TE gradient flows} which is one of the main novelties of this paper, since it improves the regularity in time for solutions of nonlocal transport equations obtained with the \emph{JKO-scheme} introduced in the following.  For clarity,  we recall that we now consider nonlocal interaction potentials $W,V$ satisfying the assumptions \emph{\hyperlink{self target}{(Self)}, \hyperlink{cross target}{(Cross)}} and we are concerned with finding a weak measure solutions in the sense of Definition \ref{def: weak solutions} for the initial value problem

\begin{equation}\label{singular-transport-equation}
\begin{cases}
\partial_t\rho(t,x)= \div_x\Big(\rho(t,x)\Big(\nabla W\ast\rho(t,x)+\nabla V\ast\nu(t,x)\Big)\Big) \ \text{ on }(0,T)\times\mathbb{R}^d,\\
\rho(0,\cdot)=\rho_0 \, \, \in \P_2(\R^d).
\end{cases}
\end{equation}

As observed in the introduction, the continuity equation \eqref{singular-transport-equation} formally shows a gradient flow structure involving a non-local interaction potential $W$ and a time-dependent external potential $\nabla V\ast\nu$. This time-dependence does not allow a proper gradient flow structure for \eqref{singular-transport-equation}.

However we can bring the well-known techniques used for such equations, see \cite{AGS,santambrogio2015optimal,carrillo2011global} to our case. More precisely, we will construct a weak solution of the transport equation \eqref{singular-transport-equation} following a suitable generalisation of the celebrated Jordan-Kinderlehrer-Otto  scheme, originally introduced in \cite{jko}, that applies to such time-depending energies $\F_{\nu(t)}$.  

The generalised JKO scheme works as follows: given the time interval $[0,T]$,  consider an uniform partition with step-size $\tau_k:= T / k$ for some $k \in \N$,  and perform the following minimisation problem
\begin{equation}\label{eq:JKO}
\tag{JKO}
\begin{cases}
\rho^k_0=\rho_0\\
\rho^k_{i+1} \in \operatorname*{argmin}\limits_{\rho\in\P_2(\R^d)}\frac{1}{2\tau_k}d_{W_2}^2(\rho,\rho_i^k)+ \F_{\nu(\tau_k(i+1))}(\rho)
\end{cases}
\end{equation}
for each $i=0, \ldots,k -1$. Note that this is formally equivalent to applying the implicit Euler scheme to \eqref{gradient-flow}. In this section we will show that the piecewise constant measures  
\begin{equation}\label{eq:pcdf}
    \rho^k(t) :=\sum\limits_{i=0}^k\rho_i^k\mathbbm{1}_{[\tau_ki,\tau_k(i+1))}(t),
\end{equation} 
will converge in some proper topology to a weak measure solution of \eqref{singular-transport-equation} (or, equivalently, \eqref{eq: TE gradient flow}) in the sense of Definition 1.1.  Moreover, we will prove that such limit solution enjoys good Lipschitz regularity in $\P_2(\R^d)$.

As first step we show that the minimization problem in \eqref{eq:JKO} is always well posed under our assumptions. For $\tau \in (0,T)$  fixed and $\bar{\mu}\in\P_2(\R^d)$ we introduce the penalised energy functional $\Phi_\tau^\nu(\bar{\mu};\cdot): \P_2(\mathbb{R}^d) \to\mathbb{R}\cup\{+\infty\}$, defined as follows
\begin{equation}\label{penalised}
 \Phi_\tau^\nu(\bar{\mu};\mu) :=\frac{1}{2\tau}d_{W_2}^2(\bar{\mu},\mu)+\F_{\nu(\tau)}(\mu).
\end{equation}

\begin{proposition}\label{minimization}
Let $W,V$ be as in \emph{\hyperlink{self target}{(Self)} and \hyperlink{cross target}{(Cross)}} respectively and let $\nu \in \Lip_{L',d_*}(0,T;\M_M^R)$ be fixed.  Let $\bar{\mu}\in\P_2(\R^d)$ be a fixed reference measure. Then, for all $\tau>0$ the minimization problem
\begin{equation}
\operatorname*{argmin}_{\mu\in P_2(\mathbb{R}^d)}\Phi_\tau^\nu(\bar{\mu};\mu) 
\end{equation}
always admits solution.
\end{proposition}
\begin{proof}
Let $\tau \in (0,T)$ be fixed and arbitrary small. In order to prove the statement we will apply the Direct Method of Calculus of Variations  showing that the minimisation problem has a solution provided that the penalised energy functional $\Phi_\tau^\nu(\bar{\mu};\cdot)$
is bounded from below,  lower semi-continuous with respect to narrow convergence and for every $\Lambda\in\R$, the level-set $$E_\Lambda=\{\mu\in P_2(\mathbb{R}^d):\Phi_\tau^\nu(\bar{\mu};\mu)\le\Lambda\}$$ is sequentially compact. We split the proof in three steps.

We start proving that $\Phi_\tau^\nu(\bar{\mu};\cdot)$ is bounded from below. 
By assumption,  we know that $\frac{1}{2\tau}d_{W_2}^2(\mu,\bar{\mu}) + \V_{\nu(\tau)}(\mu)\ge  MV_0$.  On the other hand, we claim that $W(z) \ge -B (1 + |z|^2)$ for some $B>0$.
Indeed, since the map $z\mapsto W(z)-\frac{\lambda}{2}|z|^2$ is convex and even, then it has a global minimum
and hence there exists some $A<0$ such that $W(z)-\frac{\lambda}{2}|z|^2\ge A$.  As a consequence,  
\begin{align*}
W(z)\ge - |A|-\frac{|\lambda|}{2}|z|^2 \ge -\underbrace{\max\big\lbrace|A|,\frac{|\lambda|}{2} \big\rbrace}_{:=B}(1+|z|^2)
\end{align*}
and we can estimate
\begin{align*}
\int_{\mathbb{R}^d\times\mathbb{R}^d}W(x-y)d\mu(x)d\mu(y) &\ge\int_{\mathbb{R}^d\times\mathbb{R}^d}-B(1+|x-y|^2)d\mu(x)d\mu(y)\\
& \ge -B-4\int_{\mathbb{R}^d\times\mathbb{R}^d}|x|^2d\mu(x)
\end{align*}
which is always a finite quantity since $\mu \in \P_2(\R^d)$.

The lower semi-continuity of the functional $\F_{\nu(\tau)}$ follows immediately by the continuity of $V,W$, Lemma \ref{narroConvProd} and Proposition \ref{lscOfFunctional}.  
Finally,  the lower semi-continuity of the term $\mu\mapsto \frac{1}{2\tau}d_{W_2}^2(\bar{\mu},\mu)$ follows
by Corollary \ref{lscWp}, thus the penalised energy functional $\Phi_\tau^\nu(\bar{\mu};\cdot)$ is lower semi-continuity.

To show that $E_\Lambda$ is sequentially compact,  we need to check the compactness of a sequence $(\mu_n)_n$ such that 
\begin{equation}
\Lambda\ge\int_{\mathbb{R}^d\times\mathbb{R}^d}W(x-y)d\mu_n(x)d\mu_n(y)+\int_{\mathbb{R}^d}V\ast\nu(\tau)d\mu_n+\frac{1}{2\tau}d_{W_2}^2(\bar{\mu},\mu_n).
\end{equation}
As already noticed,  assumptions \emph{\hyperlink{self target}{(Self)}} ensures that $W(z)-\frac{\lambda}{2}|z|^2$ is convex and even.  Moreover $W(z)-\frac{\lambda}{2}|z|^2\ge 0$ for every $z\in\mathbb{R}^d$,  in particular
\begin{equation}\label{help1}
W(x-y)\ge \frac{\lambda}{2}|x-y|^2\ge\lambda(|x|^2+|y|^2)
\end{equation}
since $\lambda \leq 0$.  Then, taking $\gamma_n \in \Gamma_o(\bar\mu, \mu_n)$ and using \eqref{help1}, we can estimate
\begin{align*}
\int_{\R^d \times \R^d} W(&x-y)d\mu_n(x)d\mu_n(y)+\frac{1}{2\tau}d_{W_2}^2(\bar{\mu},\mu_n)\\
=&\int_{\R^d \times \R^d} \left[W(x-y)-\lambda(|x|^2+|y|^2)\right] d\mu_n(x)d\mu_n(y)+2\lambda\int_{\R^d} |y|^2d\mu_n(y)+\frac{1}{2\tau}d_{W_2}^2(\bar{\mu},\mu_n)\\
\ge &\int_{\R^d} 2\lambda|y|^2d\mu_n(y)+\frac{1}{2\tau}d_{W_2}^2(\bar{\mu},\mu_n)\\
=& \int_{\mathbb{R}^d\times\mathbb{R}^d} \left[2\lambda|y|^2+\frac{1}{2\tau}|x-y|^2 \right]d\gamma_n(x,y) \\
\ge&\int_{\R^d \times \R^d} \left(4\lambda+ \frac{1}{2\tau} \right)|x-y|^2d\gamma_n(x,y)+4\lambda\int_{\R^d}|x|^2d\bar{\mu}(x),
\end{align*}
where the last inequality holds since
\begin{equation*}\label{help2}
2\lambda|y|^2\ge 4\lambda|x-y|^2+4\lambda|x|^2,
\end{equation*}
which can be proved combininig the fact that $|y|^2\le 2|x-y|^2+2|x|^2$ with $\lambda\le 0$.
Now observe that, for a given constant $C>0$, it is alwyas possible to choose $\tau$ small enough (i.e. smaller than $\tau<\frac{1}{2(C-4\lambda)}$) such that $4\lambda+\frac{1}{2\tau}>C$.  From the above argument we infer
\begin{equation*}
\Lambda\ge Cd_{W_2}^2(\bar{\mu},\mu_n)+4\lambda\int_{\R^d} |x|^2d\bar{\mu}+\V_{\nu(\tau)}(\mu_n),
\end{equation*}
thus 
\begin{equation*}
d_{W_2}^2(\bar{\mu},\mu_n)\le\frac{\Lambda -4\lambda\int_{\R^d} |x|^2d\bar{\mu}-\V_{\nu(\tau)}(\mu_n)}{C}
\end{equation*}
and, being $V \geq V_0$ and $\bar\mu \in \P_2(\R^d)$, we can conclude
\begin{equation}\label{eq:bdd1}
d_{W_2}^2(\bar{\mu},\mu_n)\le\frac{\Lambda-4\lambda\int_{\R^d} |x|^2d\bar{\mu}-V_0M}{C} < r
\end{equation}
for some $r>0$ depending on $\Lambda, C, V_0, M, \lambda$ and the second moment of $\bar\mu$.  In particular, the sequence $(\mu_n)_n$ belongs to the $2$-Wasserstein ball $\overline{B}_{d_{W_2}}(\bar\mu,r)$ and thanks to the observations in Remark \ref{wass-ball-tight and closed} we deduce that there is a subsequence $(\mu_{n_k})_k$ which is narrowly converging to some $\mu_* \in \overline{B}_{d_{W_2}}(\bar\mu,r)$. 
Moreover, since $\mu_n\le \Lambda$ for every $n\in\N$, by lower semicontinuity of $\Phi_\tau^\nu(\overline{\mu};\cdot)$ we get that 
\begin{equation*}
\Phi_\tau^\nu(\bar{\mu};\mu_{\ast})\le \liminf\limits_{k\to\infty}\Phi_\tau^\nu(\bar{\mu};\mu_{n_k})\le\Lambda,
\end{equation*}
hence $\mu_{\ast}\in E_\Lambda$ and the level sets of $\Phi_\tau^\nu(\bar{\mu};\cdot)$ are sequentially compact.
\end{proof}

Proposition \ref{minimization} allows us to define for every $k\ge 1$, a sequence $\rho^k_0,\rho^k_1,\ldots,\rho^k_k$ solving \eqref{eq:JKO}.  Once here we can consider piecewise constant interpolation in time of the measures $\rho^k_i$ and pass to the limit as the time step of the scheme converges to $0$.  We obtain the following compactness result. 

\begin{proposition}\label{prop: compactness JKO}
Let $T > 0$ and $\rho_0 \in \P_2(\R^d)$ be fixed.  Consider the curve $\rho^k: [0,T] \to \P_2(\R^d)$ defined in \eqref{eq:pcdf}. Then there exists a curve $\rho \in \Lip_{L,d_{W_2}}(0,T;\P_2)$ with $L = 6(M\Lip (V) + \Lip (W))$, such that, up to a non relabeled subsequence, it holds 
\[   d_n (\rho^k(t), \rho(t)) \to 0 \ \ \text{ as $k \to \infty$ for every } t \in [0,T]. \] 
\end{proposition}

\begin{proof}
For a fixed $k \in \N$, we recall the definition of $\rho^k$
\[\rho^k(t):=\sum\limits_{i=0}^{k-1}\rho_i^{k}\mathbbm{1}_{[\tau_ki,\tau_k(i+1))}(t) \]
where $\rho^k_0 = \rho_0$ and
\[  \rho^k_{i+1} \in \operatorname*{argmin}\limits_{\mu \in\P_2(\R^d)}\frac{1}{2\tau_k}d_{W_2}^2(\mu,\rho_{i}^k)+ \F_{\nu(\tau_k(i+1))}(\mu)    \quad \text{for every $i = 0 \ldots k-1$}. \]
First of all, notice that $\rho^k(t)\in \P_2(\mathbb{R}^d)$ for every $t \in [0,T]$ since $\rho^k_i \in \P_2(\mathbb{R}^d)$ for every $i = 0, \ldots, k$.
Let us now consider the $i$-th minimisation problem in \eqref{eq:JKO},  where we consider 
$\rho^k_i$ as a competitor.  We get
\begin{equation*}
\frac{1}{2\tau_k}d_{W_2}^2(\rho_i^k,\rho_{i+1}^k)+\F_{\nu(\tau_k(i+1))}(\rho_{i+1}^k) \leq 0+\F_{\nu(\tau_k(i+1))}(\rho_i^k),
\end{equation*}
which in turn gives 
\begin{equation}\label{unica importante}
\begin{aligned}
\frac{1}{2\tau_k} &d_{W_2}^2(\rho_i^k,\rho_{i+1}^k) \le \F_{\nu(\tau_k(i+1))}(\rho_i^k)-\F_{\nu(\tau_k(i+1))}(\rho_{i+1}^k)\\
=&\underbrace{\W(\rho_i^k)-\W(\rho_{i+1}^k)}_{\text{first term}}+\underbrace{\V_{\nu(\tau_k(i+1))}(\rho_i^k)-\V_{\nu(\tau_k(i+1))}(\rho_{i+1}^k)}_{\text{second term}}.
\end{aligned}
\end{equation}
Let us first focus on the \emph{first term}.  Notice that, for $\gamma_i^k\in\Gamma_o(\rho_i^k,\rho_{i+1}^k)$, we can say 
\begin{align*}
\int_{\mathbb{R}^d}W\ast\rho_i^kd\rho_i^k &-\int_{\mathbb{R}^d} W\ast\rho_{i+1}^kd\rho_{i+1}^k\\
=& \int_{\mathbb{R}^d\times\mathbb{R}^d\times\mathbb{R}^d\times\mathbb{R}^d} (W(x-y)-W(z-w))d\rho_i^k(x)d\rho_i^k(y)d\rho_{i+1}^k(z)d\rho_{i+1}^k(w)\\
\le& \int_{\mathbb{R}^d\times\mathbb{R}^d\times\mathbb{R}^d\times\mathbb{R}^d}\Lip(W)(|x-z|+|z-w|)d\gamma_i^k\otimes\gamma_i^k(x,y,z,w)\\
=&\int_{\mathbb{R}^d\times\mathbb{R}^d}\Lip(W)|x-z|d\gamma_i^k(x,z)+\int_{\mathbb{R}^d\times\mathbb{R}^d}\Lip(W)|y-w|d\gamma_i^k(y,w),
\end{align*}
and applying the Young inequality $ab\le\frac{\eps a^2}{2}+\frac{b^2}{2\eps}$ on both terms with $\eps=8\tau_k$,  we obtain 
\begin{equation}\label{eq:first term}
\begin{aligned}
 \W(\rho_i^k)-\W(\rho_{i+1}^k)
\leq& \frac{16\tau_k\Lip(W)^2}{2}+\frac{1}{16\tau_k}\int_{\mathbb{R}^d\times\mathbb{R}^d}|x-z|^2d\gamma_i^k(x,y)\\
&+\frac{1}{16\tau_k}\int_{\mathbb{R}^d\times\mathbb{R}^d}|y-w|^2d\gamma_i^k(y,w)\\
=& \,8\tau_k\Lip(W)^2+\frac{1}{8\tau_k}d_{W_2}^2(\rho_i^k,\rho_{i+1}^k).
\end{aligned}
\end{equation}

Instead,  for the \emph{second term} we get
\begin{align*}
\int_{\mathbb{R}^d}V\ast\nu(\tau_k(i&+1))d(\rho_i^k-\rho_{i+1}^k)\\
=&\int_{\mathbb{R}^d\times\mathbb{R}^d}V(x-y)d(\rho_i^k(x)-\rho_{i+1}^k(x))d\nu(\tau_k(i+1))(y)\\
=&\int_{\mathbb{R}^d}\left(\int_{\mathbb{R}^d\times\mathbb{R}^d}V(x-y)d\gamma_i^k(x,z)-\int_{\mathbb{R}^d\times\mathbb{R}^d}V(z-y)d\gamma_i^k(x,z)\right) d\nu(\tau_k(i+1))(y)\\
\le& \int_{\mathbb{R}^d}\left(\int_{\mathbb{R}^d\times\mathbb{R}^d}\Lip(V)|x-z|d\gamma_i^k(x,z)\right) d\nu(\tau_k(i+1))(y)\\
\le&\int_{\mathbb{R}^d\times\mathbb{R}^d}M\Lip(V)|x-z|d\gamma_i^k(x,z),
\end{align*}
where in the last inequality we used that $\nu(\tau_k(i+1))(\mathbb{R}^d)\le M$.  Applying again the Young inequality with $\eps = 4M \tau_k$,  we then deduce 
\begin{equation}\label{eq:second term}
    \begin{aligned}
\V_{\nu(\tau_k(i+1))}(\rho_i^k)-\V_{\nu(\tau_k(i+1))}(\rho_{i+1}^k)
\le& \int_{\mathbb{R}^d\times\mathbb{R}^d}\left(2M^2\tau_k(\Lip(V))^2+\frac{|x-z|^2}{8\tau_k}\right)d\gamma_i^k(x,z)\\
=&\, 2\tau_k (\Lip(V))^2M^2+\frac{1}{8\tau_k}d_{W_2}^2(\rho^k_i,\rho_{i+1}^k).
\end{aligned}
\end{equation}

Gathering \eqref{unica importante} together with the estimates  \eqref{eq:first term} and the \eqref{eq:second term},  we get 
\begin{equation*}
\frac{1}{2\tau_k}d_{W_2}^2(\rho_i^k,\rho_{i+1}^k)\le 2\tau_k\Lip(V)^2M^2+8\tau_k\Lip(W)^2+\frac{1}{4\tau_k}d_{W_2}^2(\rho_i^k,\rho_{i+1}^k)
\end{equation*}
which directly implies the following uniform bound on the $2$-Wasserstein distance
\begin{equation}\label{anche questa importante}
d_{W_2}(\rho_i^k,\rho_{i+1}^k) \le 6(M\Lip(V)+\Lip(W)) \tau_k.
\end{equation}
Let us observe that \eqref{anche questa importante} does not depend on the index $i$ neither on $\nu$.  Moreover, for any $t \in [0,T]$, summing \eqref{anche questa importante} for $i=0,\ldots,k-1$ we have
\begin{equation}\label{I}
 d_{W_2}(\rho_0, \rho^k(t)) \leq 6(M\Lip(V)+\Lip(W))T,
\end{equation}
which means that $\rho^k(t) \in B_{d_{W_2}}(\rho_0, r)$ with $r = 6(M\Lip(V)+\Lip(W))T$ for every $t \in [0,T]$.

Let now $0 \leq s < t \leq T$ and for each $k \in \N$ consider 
\[ s(k) := \max \{\tau_k i \leq s : i \leq k \} \quad \text{ and }  \quad t(k):= \max \{  \tau_k i \leq t  : i \leq k   \},   \]
then $s(k) \leq t(k)$ and 
\[  s = \lim_{k \to \infty} s(k) \quad \text{ and }  \quad t =  \lim_{k \to \infty} t(k).   \]
Thanks to \eqref{anche questa importante} and up to take $k$ big enough (i.e.  $\tau_k$ small enough),  we can compute 
\begin{align*}
d_{W_2}(\rho^k(s),\rho^k(t))=& d_{W_2}(\rho^k_{s(k)/\tau_k},  \rho^k_{t(k)/\tau_k}) \leq \sum_{h = s(k)/\tau_k}^{(t(k) - 1)/\tau_k} d_{W_2}(\rho^k_h, \rho^k_{h+1}) \\
& \leq 6(M\Lip(V)+\Lip(W))  \big(t(k) - s(k)\big)
\end{align*}
and passing to the limsup in $k$ we deduce 
\begin{equation}\label{II}
 \limsup_{k \to \infty} d_{W_2}(\rho^k(s),\rho^k(t)) \leq 6(M\Lip(V)+\Lip(W)) (t - s).
\end{equation}

As observed in Remark \ref{wass-ball-tight and closed} closed Wasserstein balls are sequentially compact with respect to the narrow convergence,  then thanks to \eqref{I} and \eqref{II} we can apply the refined version of Ascoli-Arzel\`a Theorem (see \cite{AGS},  Proposition 3.3.1) to conclude that there exists a further, not relabeled, subsequence $\rho^k$ and a limit curve $\rho : [0,T] \to \P_2(\R^d)$ that is $2$-Wasserstein continuous on $[0,T]$ and such that 
\[  \rho^k(t) \weak \rho(t) \ \text{ narrowly for every  } t \in [0,T].   \]

We are left to check that the limiting curve $\rho$ belongs to the space $\Lip_{L,d_{W_2}}(0,T;\P_2)$.  As observed in Corollary \ref{lscWp}, the $2$-Wasserstein distance is lower semi-continuous with respect to the narrow convergence and from \eqref{II} we deduce
\[   d_{W_2}(\rho(s),\rho(t)) \leq \liminf_{k \to \infty} d_{W_2}(\rho^k(s),\rho^k(t)) \leq \limsup_{k \to \infty} d_{W_2}(\rho^k(s),\rho^k(t))  \leq 6(M\Lip(V)+\Lip(W)) (t - s)  \]
thus concluding the proof. 
\end{proof}

We conclude this section with the proof of Theorem \ref{thm: TE gradient flows}.

\begin{proof}[Proof of Theorem \ref{thm: TE gradient flows}]

Proposition \ref{prop: compactness JKO} grants that the sequence $\rho^k$ of piecewise constant interpolations of the measures $(\rho^k_i)_{i=0}^{k-1}$ narrowly converges pointwise in time to a curve $\rho \in \Lip_{L,d_{W_2}}(0,T;\P_2)$.  We claim that $\rho$ is a weak measure solution of \eqref{eq: TE gradient flow} in the sense of Definition \ref{def: weak solutions}.  Since,  by construction,  $\rho^k(0) = \rho_0$ for every $k \in \N$, we deduce that $\rho(0) = \rho_0$.  Moreover, \[\Lip_{L,d_{W_2}}(0,T;\P_2) \subset C_{d_n}([0,T] ;\P_2(\R^d))\] and $\partial^0 W \ast \rho(t) \in L^2(\rho(t))$ for every $t \in [0,T]$.  From the Lipschitz regularity of $W$ we deduce that $\partial^0 W \ast \rho \in L^1((0,T); L^2(\rho(t)) )$. 

Therefore,  to conclude, we are only left to show that $\rho$ satisfies the weak formulation of the transport equation, namely that for every $\phi \in C^\infty_c((0,T) \times \R^d)$ it holds 
\begin{equation}\label{eq: weak form}
\int_0^T \int_{\R^d}\left(\frac{\partial\phi}{\partial t}(t,x)+   (\partial^0 W \ast \rho (t,x) + \nabla V \ast \nu (t,x))   \cdot \nabla\phi(t,x)\right)d\rho(t,x)=0.  
\end{equation}

For clarity we split the proof in four Steps. 

\textbf{Step 1.} In this Step we show that the sequence $\rho^k$ satisfies a suitable approximation of \eqref{eq: weak form} pointwise in time. This is the most classical part of the proof where, using the minimality condition in \eqref{eq:JKO} and the assumptions on the kernels, we show that for each $i = 0 \ldots k-1,$  and for every $\xi \in C^\infty_c (\R^d)$ the following identity holds 
\begin{align}\label{eq: fake weak form}
\nonumber
0=& \frac{1}{\tau_k}\int_{\R^d\times\R^d}(x-y)\cdot\nabla\xi(x)d\gamma^k_i(x,y)\\
\nonumber
&+ \frac12\int_{\R^d\times\R^d}\left[\nabla W(x-y)\cdot(\nabla\xi(x)-\nabla\xi(y))\right]d\rho^k_{i+1}(x)d\rho^k_{i+1}(y)\\
& +\int_{\R^d\times\R^d}(\nabla V(x-y)\cdot(\nabla\xi(x))d\rho^k_{i+1}(x)d\nu(\tau_k(i+1),y),
\end{align}
where $\gamma_i^{k}$ is an optimal plan between $\rho_i^k,\,  \rho_{i+1}^k$, i.e. $\gamma_i^{k} \in \Gamma_o(\rho_{i}^k,\rho_{i+1}^k)$.

Let $\xi \in C^\infty_c (\R^d)$ and $\eps >0$ be arbitrary fixed.  Recall that we constructed $\rho_{i+1}^k$ from $\rho^k_i$ by the minimization problem \eqref{eq:JKO}.  In particular, if we consider $\mu$ to be a smooth perturbation of $\rho^k_{i+1}$  i.e.  $\mu=T^{\eps}_{\#}\rho_{i+1}^k$ where $T^{\eps}(x)=x+\eps\nabla\xi(x)$,  by minimality we get 
\begin{equation*}
 \Phi_{\tau_k}^\nu(\rho^k_i;\rho_{i+1}^k)\le \Phi_{\tau_k}^\nu(\rho^k_i;T^{\eps}_{\#}\rho_{i+1}^k)
\end{equation*}
where $\Phi_{\tau_k}^\nu$ is the penalised energy functional introduced in \eqref{penalised}. Expanding the functional $\Phi_{\tau_k}^\nu$ we find
\begin{equation}\label{eq: da minimizzazione}
0\le\frac{1}{2\tau_k}\left[d_{W_2}^2(\rho_i^k,T^{\eps}_{\#}\rho^k_{i+1})-d_{W_2}^2(\rho^k_i,\rho^k_{i+1})\right]+\F_{\nu(\tau_k(i+1))}(T^{\eps}_{\#}\rho^k_{i+1})-\F_{\nu(\tau_k(i+1))}(\rho^k_{i+1}).
\end{equation}

Recalling that $\gamma_i^{k}\in \Gamma_o(\rho_{i}^k,\rho_{i+1}^k)$,  we can estimate the terms in \eqref{eq: da minimizzazione} containing the Wasserstein distance in the following way
\begin{equation}\label{wass-term}
    \begin{aligned}
\frac{1}{2\tau_k}\Big[d_{W_2}^2(\rho_i^k,T^{\eps}_{\#}\rho^k_{i+1}) & -d_{W_2}^2(\rho^k_i,\rho^k_{i+1})\Big]\\
\le&\frac{1}{2\tau}\int_{\mathbb{R}^d\times\mathbb{R}^d}|x-y|^2d(T^\eps,Id)_{\#}\gamma_i^k(x,y)-\frac{1}{2\tau}\int_{\mathbb{R}^d\times\mathbb{R}^d}|x-y|^2d\gamma_i^k\\
=& \frac{1}{2\tau_k}\int_{\R^d\times\R^d}|T(x)-y|^2d\gamma_i^k(x,y)-\frac{1}{2\tau_k}\int_{\R^d\times\R^d}|x-y|^2d\gamma_i^k(x,y)\\
=&\frac{1}{2\tau_k}\int_{\R^d\times\R^d}(|x+\eps\nabla\xi(x)-y|^2-|x-y|^2)d\gamma_i^k(x,y).
\end{aligned}
\end{equation}

Notice,  that in the first inequality of \eqref{wass-term} we used the competitor plan $(T^\eps,Id)_{\#}\gamma_i^k\in\Gamma(\mu,\rho_i^k)$ where $(T^\eps,Id): (x,y)\mapsto (T^\eps(x),y)$.
Indeed, since
\begin{equation*}
\pi_1\circ(T^\eps,Id)(x,y)=T^\eps(x)=T^\eps\circ\pi_1(x,y) \quad \text{ and } \quad \pi_2\circ(T^\eps,Id)(x,y)=y=\pi_2(x,y),
\end{equation*}
we have
\begin{align*}
(\pi_1)_{\#}(T^\eps,Id)_{\#}\gamma_i^k=(\pi_1\circ(T^\eps,Id))_{\#}\gamma_i^k=(T^\eps\circ\pi_1)_{\#}\gamma_i^k =T^\eps_{\#}\underbrace{(\pi_1)_{\#}\gamma_i^k}_{\rho_{i+1}} =\mu,
\end{align*}
and also
\begin{equation*}
(\pi_2)_{\#}(T^\eps,Id)_{\#}\gamma_i^k=(\pi_2)_{\#}\gamma_i^k=\rho^k_i.
\end{equation*}

Let us now discuss the terms in \eqref{eq: da minimizzazione} involving the energy $\F_{\nu(\tau_k(i+1))}$.  From the definition of $T^\eps$ we can directly compute 
\begin{align*}
\F_{\nu(\tau_k(i+1))}(T^{\eps}_{\#}\rho^k_{i+1})&-\F_{\nu(\tau_k(i+1))}(\rho^k_{i+1})\\
=&\frac12\int_{\R^d \times \R^d} \left[W(x-y+\eps(\nabla\xi(x)-\nabla\xi(y))-W(x-y)\right]d\rho^k_{i+1}(x)d\rho^k_{i+1}(y)\\
& +\int_{\R^d \times \R^d} \left[V(x-y+\eps\nabla\xi(x))-V(x-y)\right]d\rho^k_{i+1}(x)d\nu(\tau_k(i+1),y).
\end{align*}
Gathering together the above identity,  \eqref{wass-term} and \eqref{eq: da minimizzazione},  we get
\begin{equation}\label{non so come chiamarla}
    \begin{aligned}
0\le& \frac{1}{2\tau_k}\int_{\R^d\times\R^d}(|x-y+\eps\nabla\xi(x)|^2-|x-y|^2)d\gamma^k_i(x,y)\\
& +\frac12 \int_{\R^d\times\R^d}\left[W(x-y+\eps(\nabla\xi(x)-\nabla\xi(y))-W(x-y)\right]d\rho^k_{i+1}(x)d\rho^k_{i+1}(y)\\
& + \int_{\R^d\times\R^d}(V(x-y+\eps\nabla\xi(x))-V(x-y))d\rho^k_{i+1}(x)d\nu(\tau_k(i+1),y)
\end{aligned}
\end{equation}

In order to obtain \eqref{eq: fake weak form},  we need to divide \eqref{non so come chiamarla} by $\eps$ and pass to the limit separately in the three terms of the r.h.s.  For more clarity,  we denote
\begin{align*}
I &:= \frac{1}{2\tau_k}\int_{\R^d\times\R^d}(|x-y+\eps\nabla\xi(x)|^2-|x-y|^2)d\gamma^k_i(x,y), \\
II &:= \frac12 \int_{\R^d\times\R^d}\left[W(x-y+\eps(\nabla\xi(x)-\nabla\xi(y))-W(x-y)\right]d\rho^k_{i+1}(x)d\rho^k_{i+1}(y), \\
III &:= \int_{\R^d\times\R^d}(V(x-y+\eps\nabla\xi(x))-V(x-y))d\rho^k_{i+1}(x)d\nu(\tau_k(i+1),y).
\end{align*}
Let us start with $I$.  Whenever $x \neq y$ we can find some $\bar{x} \in \spt(\xi)$ such that
\[
\left| \frac{|x-y+\eps\nabla\xi(x)|^2-|x-y|^2}{\eps} \right| \leq  2|x-y|\|\nabla\xi\|_{L^{\infty}}+|x-y|^2\|D^2\xi\|_{L^{\infty}}
\]
and the latter  is in $L^{1}(\gamma^k_i)$ since $\rho^k_i,\rho^k_{i+1}\in\P_2(\R^d)$. Then the limit as $\eps \to 0$ is well defined and corresponds to 
we deduce that 
\begin{equation}\label{limit of I}
 \begin{aligned}
\lim_{\eps \searrow 0} \,\frac{1}{2\tau_k}\int_{\R^d\times\R^d} & \frac{|x-y+\eps  \nabla\xi(x)|^2-|x-y|^2}{\eps}d\gamma^k_i(x,y) \\
 &=  \frac{1}{\tau_k}\int_{\R^d\times\R^d} (x-y) \cdot \nabla \xi(x) d\gamma^k_i(x,y).
\end{aligned}   
\end{equation}
We can deal with term $III$ similarly as done before for $I$.  Indeed,  since $V \in C^1(\R^d)$, we get 
\[
\left|\frac{V(x-y+\eps\nabla\xi(x))-V(x-y)}{\eps}\right| \le 2\Lip(V)\|\nabla\xi\|_{L^{\infty}(\R^d)}
\]
and the latter is in $L^1(\rho^k_{i+1}\otimes\nu(\tau_k(i+1)))$ since $\rho^k_{i+1}\otimes\nu(\tau_k(i+1))$ is a finite measure.  Then
\begin{equation}\label{limit of III}
  \begin{aligned}
\lim_{\eps \searrow 0} \,\int_{\R^d\times\R^d} & \frac{V(x-y+\eps\nabla\xi(x))-V(x-y)}{\eps}d\rho^k_{i+1}(x)d\nu(\tau_k(i+1),y)  \\
 &=  \int_{\R^d\times\R^d} \nabla V(x-y) \cdot \nabla \xi(x) d\rho^k_{i+1}(x) d\nu(\tau_k(i+1),y).
\end{aligned}  
\end{equation}
For what concerns term $II$,  it is immediate to prove the following pointwise convergence 
\[
\frac{W(x-y+\eps(\nabla\xi(x)-\nabla\xi(y)))-W(x-y)}{\eps} \to \nabla W(x-y)\cdot (\nabla\xi(x)-\nabla\xi(y))),
\]
everywhere in $\R^d \times \R^d$. Since
\begin{align*}
\left|\frac{W(x-y+\eps(\nabla\xi(x)-\nabla\xi(y)))-W(x-y)}{\eps}\right|\le 2\Lip(W)\|\nabla\xi\|_{L^{\infty}(\R^d)},
\end{align*}
and the r.h.s.  is in $L^1(\rho^k_{i+1}\otimes\rho^k_{i+1})$, we obtain
\begin{equation}\label{limit of II}
\begin{aligned}
\lim\limits_{\eps\searrow 0} \,\int_{\R^d\times\R^d} & \frac{W(x-y+\eps(\nabla\xi(x)-\nabla\xi(y)))-W(x-y)}{\eps}d\rho^k_{i+1}(x)d\rho^k_{i+1}(y)\\
&= \int_{\R^d\times\R^d}\nabla W(x-y)\cdot(\nabla\xi(x)-\nabla\xi(y)))d\rho^k_{i+1}(x)d\rho^k_{i+1}(y).
\end{aligned}
\end{equation}
%
Dividing  \eqref{non so come chiamarla} by $\eps >0$ and taking the limit as $\eps\searrow0$,  from \eqref{limit of I}, \eqref{limit of III} and \eqref{limit of II} we obtain that
\begin{align*}
0\le& \frac{1}{\tau_k}\int_{\R^d\times\R^d}(x-y)\cdot\nabla\xi(x)d\gamma^k_i(x,y)\\
&+ \frac12 \int_{\R^d\times\R^d}\left[\nabla W(x-y)\cdot(\nabla\xi(x)-\nabla\xi(y))\right]d\rho^k_{i+1}(x)d\rho^k_{i+1}(y)\\
&+\int_{\R^d\times\R^d}(\nabla V(x-y)\cdot(\nabla\xi(x))d\rho^k_{i+1}(x)d\nu(\tau_k(i+1),y),
\end{align*}
and repeating the same argument with $\eps<0$ this time sending $\eps\nearrow 0$, we obtain the reverse inequality which implies \eqref{eq: fake weak form}. 

\textbf{Step 2.}
In this Step we show that the sequence $\rho^k$ satisfies the approximation of \eqref{eq: weak form} for test functions which are piecewise constant in the time variable.  More precisely,  we will show that for every $\xi \in C^\infty_c(\R^d)$ and 
\begin{equation}\label{thetak}
\theta^k(t)=\sum\limits_{i=0}^{k-1}\theta(i\tau_k)\mathbbm{1}_{[\tau_k i,\tau_k(i+1))} \mbox{ for some values } \{\theta(i\tau_k) \}_{i=0 \ldots k-1},     
\end{equation} 
the following inequalities hold
\begin{equation}\label{befor-limit}
    \begin{aligned}
\int_{\mathbb{R}^d}\xi(x)\theta^k(t)& d\rho^k(t,x) -\int_{\mathbb{R}^d}\xi(x)\theta^k(s)d\rho^k(s,x)\\
\le &\sum\limits_{h=j}^{i-1}\frac12\|\theta\|_{L^{\infty}}\|D^2\xi\|_{L^{\infty}}d_{W_2}^2(\rho^k_{h+1},\rho^k_h) + \mathcal{O}(\tau_k)\\
& +\frac12\int_s^t\int_{\mathbb{R}^d\times\mathbb{R}^d}\theta^k(q)
\nabla W(x-y)\cdot(\nabla\xi(x)-\nabla\xi(y))d\rho^k(q,x)\otimes\rho^k(q,y)\\
& +\int_s^t\int_{\mathbb{R}^d\times\mathbb{R}^d}\theta^k(q)\nabla V(x-y)\cdot\nabla\xi(x)d\rho^k(q,x)\otimes\nu^k(q,y)\\
& +\sum\limits_{h=j}^{i-1}(\theta(\tau_k(h+1))-\theta(\tau_k h))\int_{\mathbb{R}^d}\xi(x)d\rho_h^k(x)
\end{aligned}
\end{equation}
and
\begin{equation}\label{befor-limit_with_minus}
\begin{aligned}
\int_{\mathbb{R}^d}\xi(x)\theta^k(t) & d\rho^k(t,x) -\int_{\mathbb{R}^d}\xi(x)\theta^k(s)d\rho^k(s,x)\\
\geq & -\sum\limits_{h=j}^{i-1}\frac12\|\theta\|_{L^{\infty}}\|D^2\xi\|_{L^{\infty}}d_{W_2}^2(\rho^k_{h+1},\rho^k_h) + \mathcal{O}(\tau_k)\\
& +\frac12\int_s^t\int_{\mathbb{R}^d\times\mathbb{R}^d}\theta^k(q)
\nabla W(x-y)\cdot(\nabla\xi(x)-\nabla\xi(y))d\rho^k(q,x)\otimes\rho^k(q,y)\\
& +\int_s^t\int_{\mathbb{R}^d\times\mathbb{R}^d}\theta^k(q)\nabla V(x-y)\cdot\nabla\xi(x)d\rho^k(q,x)\otimes\nu^k(q,y)\\
& +\sum\limits_{h=j}^{i-1}(\theta(\tau_k(h+1))-\theta(\tau_k h))\int_{\mathbb{R}^d}\xi(x)d\rho_h^k(x)
\end{aligned}
\end{equation}
for every $s \leq t$ in $[0,T]$,  where $s \in [\tau_k i,\tau_k(i+1))$ and $s \in [\tau_k j,\tau_k(j+1))$ for some $i \leq j$, and  $\nu^k: [0,T] \to \M_M^R(\R^d)$ corresponds to the piecewise constant interpolation of $\nu$ defined as 
\[ \nu^k(t) :=  \sum\limits_{i=0}^{k-1} \nu(\tau_k(i+1)) \mathbbm{1}_{[\tau_k i,\tau_k(i+1))}(t). \]

We only prove \eqref{befor-limit}, since \eqref{befor-limit_with_minus} follows from a similar argument. Let $\theta^k(t)$ and $\xi \in C^\infty_c(\R^d$ be arbitrary fixed and let $i \in \{0, \dots k-1\}$.   Multiplying \eqref{eq: fake weak form} by $\theta(\tau_k(i+1))$ we obtain 
\begin{equation}\label{eq:zero}
    \begin{aligned}
0=&\theta(\tau_k(i+1))\int_{\mathbb{R}^d\times\mathbb{R}^d}(x-y)\cdot\nabla\xi(x)d\gamma_i^k(x,y)\\
+&\frac{\tau_k}{2}\theta(\tau_k(i+1))\int_{\mathbb{R}^d\times\mathbb{R}^d}
\nabla W(x-y)\cdot(\nabla\xi(x)-\nabla\xi(y))d\rho_{i+1}^k(x)d\rho_{i+1}^k(y)\\
+&\tau_k \theta(\tau_k(i+1))\int_{\mathbb{R}^d\times\mathbb{R}^d}\nabla V(x-y)\cdot\nabla\xi(x)d\rho_{i+1}^k(x)d\nu(\tau(i+1),y).
\end{aligned}
\end{equation}
From a second order Taylor expansion on $\xi$ 
we deduce
\begin{align*}
\theta(\tau_k(i+1))\int_{\mathbb{R}^d\times\mathbb{R}^d} & \left(\xi(y)-\xi(x)+(x-y)\cdot\nabla\xi(x)\right) d\gamma_i(x,y)\\
=\,&\theta(\tau_k(i+1))\int_{\mathbb{R}^d\times\mathbb{R}^d}\frac12(x-y)^T D^2\xi(\bar{x})(x-y)d\gamma_i^k(x,y),
\end{align*}
for a proper $\bar{x}$ in $\R^d$. Thus
\begin{equation}\label{modulus}
\begin{aligned}
&\left|\theta(\tau_k(i+1))\left(\int_{\mathbb{R}^d}  \xi(x)d\rho_{i+1}^k(x)-\int_{\mathbb{R}^d}\xi(x)d\rho_i^k(x)+\int_{\mathbb{R}^d\times\mathbb{R}^d}(x-y)\cdot\nabla\xi(x)d\gamma_i^k(x,y) \right) \right|\\
\leq & \frac12\|\theta\|_{L^{\infty}}\|D^2\xi\|_{L^{\infty}}d_{W_2}^2(\rho^k_{i+1},\rho^k_i).
\end{aligned}
\end{equation}
By applying the above inequality to \eqref{eq:zero} we get
\begin{equation}\label{eq:succesiveIncrements}
\begin{aligned}
\theta(\tau_k(i+1))&\left(\int_{\mathbb{R}^d}\xi(x)d\rho_{i+1}^k(x)-\int_{\mathbb{R}^d}\xi(x)d\rho_i^k(x)\right)\\
\leq& \frac12\|\theta\|_{L^{\infty}}\|D^2\xi\|_{L^{\infty}}d_{W_2}^2(\rho^k_{i+1},\rho^k_i)\\
+&\frac{\tau_k}{2}\theta(\tau_k(i+1))\int_{\mathbb{R}^d\times\mathbb{R}^d}
\nabla W(x-y)\cdot(\nabla\xi(x)-\nabla\xi(y))d\rho_{i+1}^k(x)d\rho_{i+1}^k(y)\\
+&\theta(\tau_k(i+1))\int_{\mathbb{R}^d\times\mathbb{R}^d}\nabla V(x-y)\cdot\nabla\xi(x)d\rho_{i+1}^k(x)d\nu(\tau_k(i+1),y).
\end{aligned}    
\end{equation}
Let now $s\leq t$, where $s\in[\tau_k i,\tau_k(i+1))$ and $t\in [\tau_k j, \tau_k(j+1))$ for some $i \leq j$,  then by a telescopic sum we can compute
\begin{align*}
\int_{\mathbb{R}^d}\xi(x)\theta^k(t)d\rho^k(t,x)&-\int_{\mathbb{R}^d}\xi(x)\theta^k(s)d\rho^k(s,x) \\
=&\, \theta(j\tau_k)\int_{\mathbb{R}^d}\xi(x)d\rho^k_j(x)-\theta(i\tau_k)\int_{\mathbb{R}^d}\xi(x)d\rho^k_i(x) \\
=&\,\sum\limits_{h=i}^{j-1}\theta((h+1)\tau_k)\left(\int_{\mathbb{R}^d}\xi(x)d\rho^k_{h+1}(x)-\int_{\mathbb{R}^d}\xi(x)d\rho_h(x)\right)\\
&+\sum\limits_{h=i}^{j-1}(\theta((h+1)\tau_k)-\theta(h\tau_k))\int_{\mathbb{R}^d}\xi(x)d\rho_h^k(x),
\end{align*}
and thanks to \eqref{eq:succesiveIncrements} we further deduce that
\begin{align*}
\int_{\mathbb{R}^d}\xi(x)\theta^k(t)d\rho^k(t,x) & -\int_{\mathbb{R}^d}\xi(x)\theta^k(s)d\rho^k(s,x)\\
\le& \frac12 \sum\limits_{h=i}^{j-1} \|\theta\|_{L^{\infty}}\|D^2\xi\|_{L^{\infty}}d_{W_2}^2(\rho^k_{h+1},\rho^k_h)\\
+& \frac{\tau_k}{2} \sum\limits_{h=i}^{j-1}\theta(\tau_k(h+1))\int_{\mathbb{R}^d\times\mathbb{R}^d}
\nabla W(x-y)\cdot(\nabla\xi(x)-\nabla\xi(y))d\rho_{h+1}^k(x) d\rho_{h+1}^k(y)\\
+& \tau_k\sum\limits_{h=i}^{j-1}\theta(\tau_k(h+1))\int_{\mathbb{R}^d\times\mathbb{R}^d}\nabla V(x-y)\cdot\nabla\xi(x)d\rho_{h+1}^k(x)d\nu^k(\tau_k(h+1),y)\\
+&\sum\limits_{h=i}^{j-1}(\theta(\tau_k(h+1))-\theta(\tau_k h))\int_{\mathbb{R}^d}\xi(x)d\rho_h^k(x).
\end{align*}
Finally,  recalling the piecewise constant structure of $\theta^k$ and $\nu^k$ and the fact that $t -s = (j - i)\tau_k + \mathcal{O}(\tau_k)$,  it is immediate to deduce \eqref{befor-limit} from the above inequality. Note that \eqref{befor-limit_with_minus} easily follows from a similar argument but applying the reverse inequality of \eqref{modulus} in \eqref{eq:succesiveIncrements}.

\textbf{Step 3.} In this Step we consider a generic test function $\theta \in C^\infty_c(0,T)$,  its corresponding piecewise constant approximation $\theta^k(t)$ defined in \eqref{thetak} and we pass to the limit in \eqref{befor-limit} when $k \to \infty$,  or,  equivalently,  when $\tau_k \to 0$. 
For clarity,  we analyse separately each term involved in  \eqref{befor-limit}-\eqref{befor-limit_with_minus}.

The convergence 
\[
\lim_{k\to \infty} \theta^k(t)\int_{\R^d}\xi(x)d\rho^k(t,x) = \theta(t)\int_{\R^d}\xi(x)d\rho(t,x)
\]
is a straightforward consequence of the narrow convergence $\rho^k\weak\rho$ and the fact that $|\theta^k(q)-\theta(q)|\le \|\theta'\|_{L^{\infty}(\R^d)}\tau_k$.  The convergence
\begin{equation*} 
\lim_{k\to \infty}  \frac12\sum\limits_{h=i}^{j-1} \|\theta\|_{L^\infty([0,T])}\|D^2\xi\|_{L^{\infty}(\R^d)}d_{W_2}^2(\rho^k_{h+1},\rho^k_h) + \mathcal{O}(\tau_k) =  0 
\end{equation*}
is granted by Proposition \ref{prop: compactness JKO},  indeed  
\[
\limsup_{k \to \infty} \sum\limits_{h=i}^{j-1} d_{W_2}^2(\rho^k_{h+1},\rho^k_h)= \limsup_{k \to \infty} \sum\limits_{h=i}^{j-1} d_{W_2}^2(\rho^k(h\tau_k),\rho^k((h+1)\tau_k)) \leq \sum\limits_{h=i}^{j-1} L^2\tau_k^2 \le L^2 \tau_k.
\]
In order to show that
\begin{align*}
\lim_{k \to \infty} &\frac12\int_s^t\int_{\mathbb{R}^d\times\mathbb{R}^d}\theta^k(q)
\nabla W(x-y)\cdot(\nabla\xi(x)-\nabla\xi(y))d\rho^k(q,x)\otimes\rho^k(q,y)\\
&= \frac12\int_s^t\int_{\mathbb{R}^d\times\mathbb{R}^d}\theta(q)
\nabla W(x-y)\cdot(\nabla\xi(x)-\nabla\xi(y))d\rho(q,x)\otimes\rho(q,y)
\end{align*}
is enough to notice that it is instead equivalent to 
\begin{align*}
\lim_{k \to \infty} &\frac12\int_s^t\int_{\mathbb{R}^d\times\mathbb{R}^d}(\theta^k(q)-\theta(q))
\nabla W(x-y)\cdot(\nabla\xi(x)-\nabla\xi(y))d\rho^k(q,x)\otimes\rho^k(q,y)\\
& +\frac12\int_s^t\int_{\mathbb{R}^d\times\mathbb{R}^d}\theta(q)
\nabla W(x-y)\cdot(\nabla\xi(x)-\nabla\xi(y))d(\rho^k\otimes\rho^k-\rho\otimes\rho) = 0,
\end{align*}
and the latter follows by applying Lemma \ref{narroConvProd} and recalling that $|\theta^k(q)-\theta(q)|\le \|\theta'\|_{L^{\infty}(\R^d)}\tau_k$ and $\nabla W(x-y)\cdot(\nabla\xi(x)-\nabla\xi(y))$ is continuous and bounded. By a similar argument as the one above
\begin{align*}
\lim_{k \to \infty}&\int_s^t\int_{\mathbb{R}^d\times\mathbb{R}^d}\theta^k(q)\nabla V(x-y)\cdot\nabla\xi(x)d\rho^k(q,x)\otimes\nu^k(q,y)\\
&=\int_s^t\int_{\mathbb{R}^d\times\mathbb{R}^d}\theta(q)\nabla V(x-y)\cdot\nabla\xi(x)d\rho(q,x)\otimes\nu(q,y),
\end{align*}
holds recalling that $\nabla V$ is continuous by assumption,  moreover Lemma \ref{lem: piecewise interpolation} and Lemma \ref{narroConvProd} ensure that $\rho^k \otimes \nu^k\weak\rho\otimes\nu$ narrowly. We are left to show that
\begin{equation*}
\lim_{k \to \infty} \sum\limits_{h=i}^{j-1} \left(\theta(\tau_k(h+1))-\theta(\tau_k h)\right)\int_{\R^d}\xi(x)d\rho^k_h(x) =  \int_s^t \theta'(q)\int_{\R^d}\xi(x)d\rho(q,x)dq
\end{equation*}
By first order expansion,  we can find some $\eta_h\in (\tau_k h,\tau_k(h+1))$ such that
\begin{align*}
\theta(\tau_k(h+1))-\theta(\tau_k h)) =&\theta'(\tau_k h)\tau_k+\theta''(\eta_h)\frac{\tau_k^2}{2}.
\end{align*}
In particular
\begin{equation}\label{approx1}
\begin{aligned}
&\left|\sum\limits_{h=i}^{j-1}\left(\theta(\tau_k(h+1))-\theta(\tau_k h)\right)\int_{\R^d}\xi(x)d\rho^k_h(x)-\sum\limits_{h=i}^{j-1}\tau_k\theta'(\tau_k h)\int_{\R^d}\xi(x)d\rho^k_h(x)\right|\\
&\le \sum\limits_{h=i}^{j-1} \|\theta''\|_{L^{\infty}(\R^d)}\frac{\tau_k^2}{2}\|\xi\|_{L^{\infty}(\R^d)} \le\frac{\tau_k}{2}\|\theta\|_{L^{\infty}(\R^d)}\|\xi\|_{L^{\infty}(\R^d)}, 
\end{aligned}
\end{equation}
and 
\begin{equation}\label{approx2}
    \begin{aligned}
&\sum\limits_{h=i}^{j-1}\tau_k\theta'(\tau_k h)\int_{R^d}\xi(x)d\rho^k_h(x)\\
=&\sum\limits_{h=i}^{j-1}\left[\int_{\tau_k h}^{\tau_k(h+1)}(\theta'(\tau_k h)-\theta'(q))dq+\int_{\tau_k h}^{\tau_k(h+1)}\theta'(q)dq\right]\int_{\R^d}\xi(x)d\rho^k_h(x)\\
\le& \tau_k\|\theta''\|_{L^{\infty(\R^d)}}\|\xi\|_{L^{\infty(\R^d)}}+\int_{\tau_k i}^{\tau_k j}\theta'(q)\int_{\R^d}\xi(x)d\rho^k(q,x)dq.
\end{aligned}
\end{equation}
Then we conclude passing to the limit as $k \to \infty$ in \eqref{approx1}, \eqref{approx2} and recalling that $\tau_k i \to s$,  $\tau_k j \to t$ and $\theta' \xi \in C_b((0,T) \times \R^d)$.

Summarizing,  passing to the limit as $k \to \infty$ in \eqref{befor-limit}-\eqref{befor-limit_with_minus}, we obtain the following identity
\begin{equation}\label{eq:almostWeakForm}
\begin{aligned}
\int_{\mathbb{R}^d} \xi(x)\theta(t)d\rho(t,x) &-\int_{\mathbb{R}^d}\xi(x)\theta(s)d\rho(s,x)\\
=&\frac12\int_s^t\int_{\mathbb{R}^d\times\mathbb{R}^d}\theta(q)\nabla W(x-y)\cdot (\nabla\xi(x)-\nabla\xi(y))d\rho(q,x)\otimes\rho(q,y) \\ 
&+\int_s^t\int_{\mathbb{R}^d\times\mathbb{R}^d}\theta(q)\nabla V(x-y)\cdot \nabla\xi(x)d\rho(q,x)\otimes\nu(q,y) \\
&+\int_s^t\int_{\mathbb{R}^d}\theta'(q)\xi(x)d\rho(q,x).
\end{aligned}
\end{equation}

\textbf{Step 4.} In this Step we finally prove that $\rho$ satisfies \eqref{eq: weak form} for every $\phi \in C^\infty_c((0,T) \times \R^d)$.  This will come as a straightforward consequence of \eqref{eq:almostWeakForm}. 
Indeed, by a density argument a test function $\phi \in C^\infty_c((0,T) \times \R^d)$ can be approximated by $\theta \in  C^\infty_c(0,T)$ and $\xi \in  C^\infty_c(\R^d)$,  applying \eqref{eq:almostWeakForm} with $t=T$ and $s=0$ we deduce 

\begin{align*}
\int_{\mathbb{R}^d}\xi(x)\theta(T)d\rho(T,x) &-\int_{\mathbb{R}^d}\xi(x)\theta(0)d\rho(0,x)-\int_0^T\int_{\mathbb{R}^d}\theta'(q)\xi(x)d\rho(q,x)\\
=&\int_s^t\theta(q)\frac{d}{dq}\left[\int_{\mathbb{R}^d}\xi(x)d\rho(q,x)\right]dq =  -\int_0^T\theta'(q)\left[\int_{\R^d}\xi(x)d\rho(q,x)\right]dq. 
\end{align*}

Moreover, for every $q \in [0,T]$ we have
\[  \int_{\mathbb{R}^d\times\mathbb{R}^d} \nabla V(x-y)\cdot \nabla\xi(x) d\rho(q,x) \otimes \nu(q,y) = \int_{\R^d} \nabla V \ast \nu (q,x) d\rho(q,x) \]
and, being $W$ symmetric, also 
\[
\frac12\int_{\mathbb{R}^d\times\mathbb{R}^d} \nabla W(x-y)\cdot(\nabla\xi(x)-\nabla\xi(y))d\rho(q,x)\otimes\rho(q,y) =  \int_{\R^d} \partial^0 W \ast \rho(q,x) \cdot \nabla \xi(x) d\rho(q,x),
\]
thus \eqref{eq: weak form} follows.
\end{proof}

\section{Optimal control problem}\label{optimal-control-section}

This section is devoted to the proof of Theorem \ref{thm: OC gradient flows}.  We will show that problem \eqref{eq: OC gradient flow} admits a solution in the class 
\[ A = \Lip_{L',d_*}(0,T;\M_M^R) \, \times \,  \Lip_{L,d_{W_2}}(0,T;\P_2), \]
whenever $W,V$ are as in \emph{\hyperlink{self target}{(Self)},  \hyperlink{cross target}{(Cross)}}, and in the larger class 
\[   A' = \Lip_{L',d_*}(0,T;\M_M^R) \, \times \,  C_{d_n}(0,T;\P_2), \]
in case $V$ is also $\lambda'$-convex for some $\lambda' \leq 0$.

\begin{proof}[Proof of Theorem \ref{thm: OC gradient flows}]

Let us first assume that $W,V$ are as in \emph{\hyperlink{self target}{(Self)},  \hyperlink{cross target}{(Cross)}} respectively. Then the minimization problem \eqref{eq: OC gradient flow} can be formulated 
in the following equivalent way 
\[ \inf_A  \big( \J(\nu,\rho) + \chi_B(\nu,\rho) \big)   \]
where 
\[ B: = A \cap  \{ (\nu,\rho) : \ \rho \text{ is a weak measure solution of \eqref{eq: TE gradient flow} with $\nu$ and initial datum }  \rho(0) = \rho_0\},\]
and $\chi_B$ is the standard characteristic function of the set $B$,  i.e. 
\[  \chi_B(\nu,\rho) = \begin{cases}
0  &\text { if } (\nu,\rho) \in B \\  +\infty & \text{ otherwise.}
\end{cases}  \]

Without loss of generality, we can assume that $\big((\nu_k,\rho_k)\big)_k \subset A$ is a minimizing sequence 
for $\J + \chi_B$ satisfying 
\[  \mathcal{J}(\nu_k,\rho_k)+\chi_B(\nu_k,\rho_k)  < + \infty  \ \  \text{ and }  \ \  \lim_{k \to \infty}  \big( \mathcal{J}(\nu_k,\rho_k)+\chi_B(\nu_k,\rho_k)  \big) =  \inf_A \big(\J(\nu,\rho) + \chi_B(\nu,\rho)\big). \]
In particular,  for every $k \in \N$ the curve $\rho_k$ is a weak measure solution of \eqref{eq: TE gradient flow} with $\nu_k$ and initial datum $\rho_0$ in the sense of Definition \ref{def: weak solutions}. 

We show that the two sequences $(\rho_k)_k$ and $(\nu_k)_k$ independently enjoy good compactness properties with respect to the narrow convergence of measures. 

Let us first focus on the compactness of $(\rho_k)_k$.  By definition we have that $(\rho_k)_k \subset  \Lip_{L,d_{W_2}}(0,T;\P_2)$ and hence $\rho_k(t) \in B_{d_{W_2}}(\rho_0, LT)$  for every $k \in \N$ and $t \in [0,T]$. 
Then,  arguing as in the proof of Proposition \ref{prop: compactness JKO},  we can apply Ascoli-Arzel\`a Theorem to deduce the existence of a $2$-Wasserstein continuous curve $\rho : [0,T] \to \P_2(\R^d)$ and a not relabeled subsequence $\rho_k$ such that $d_n (\rho_k(t),\rho(t)) \to 0$ for every $t \in [0,T]$.  Moreover, recalling that Wasserstein distances are lower semi-continous with respect to the narrow convergence of measures,  it is immediate to observe that $\rho \in \Lip_{L,d_{W_2}}(0,T;\P_2)$. 

The compactness of $(\nu_k)_k$ follows also by a standard application of Ascoli-Arzel\`a Theorem.  
Indeed,  with respect to the weak-$*$ topology,  for each $t \in [0,T]$ the sequence $(\nu_k(t))_k$ is relatively compact in $\M_M^R(\R^d)$ (which is itself compact in $\M(\R^d)$ thanks to Banach-Alaoglu Theorem). Moreover,  $(\nu_k)_k$ is equi-Lipschitz (with constant $L'$) on $[0,T]$.  Therefore,  there exists a limit measure $\nu : [0,T] \to \M_M^R(\R^d)$ that is $L'$-Lipschitz continuous with respect to the weak-$*$ topology and such that,  up to subsequences,  $d_* (\nu_k(t), \nu(t)) \to 0$ for every $t \in [0,T]$. 
Moreover,  since $\spt(\nu_k(t)) \subset \overline{B(0,R)}$,  the sequence $(\nu_k(t))_k$ is tight and hence,  by Prokhorov's Theorem,  we infer that $d_n (\nu_k(t), \nu(t)) \to 0$ for every $t \in [0,T]$. 

Since $\J + \chi_B$ is pointwise (in time) lower semi-continuous with respect to $d_n$,  to conclude that \eqref{eq: OC gradient flow} admits solution in $A$, we are left to show that $\rho$ is a weak measure solution \eqref{eq: TE gradient flow} with $\nu$ and initial datum $\rho_0$ in the sense of Definition \ref{def: weak solutions}.

Since, by construction,  any $\rho_k$ is a weak measure solution of \eqref{eq: TE gradient flow} with $\nu_k$ and initial datum $\rho_0$,  we only need to check that for every $\phi \in C^\infty_c((0,T) \times \R^d)$ it holds 
\begin{align*}
 \lim_{k \to \infty} \int_0^T \int_{\R^d}\left(\frac{\partial\phi}{\partial t}(t,x)+   (\partial^0 W \ast \rho_k (t,x) + \nabla V \ast \nu_k (t,x))   \cdot \nabla\phi(t,x)\right)d\rho_k (t,x) \\
 = \int_0^T \int_{\R^d}\left(\frac{\partial\phi}{\partial t}(t,x)+   (\partial^0 W \ast \rho (t,x) + \nabla V \ast \nu (t,x))   \cdot \nabla\phi(t,x)\right)d\rho (t,x).
\end{align*}

The convergence of  term involving the time derivative and the one involving $\nabla V$ is immediate thanks to the regularity of $\frac{\partial\phi}{\partial t}$ and $\nabla V$ and the claim of Proposition \ref{narroConvProd}, which ensures that $d_n (\nu_k(t) \otimes \rho_k(t)) \to 0$ for every $t \in [0,T]$ on the product space $\R^d \times \R^d$. 

On the other hand,  as already observed in the proof of Theorem \ref{thm: TE gradient flows},  the symmetry of $W$ and the definition of $\partial^0 W $  imply that 

\begin{align*}
\int_{\R^d} \partial^0 W \ast \rho_k (t,x) & \cdot \nabla\phi(t,x) d\rho_k (t, x)  \\
&= \frac{1}{2} \int_{\R^d} \int_{\R^d} \nabla W(x-y) \cdot \big( \nabla \phi(t,x) - \nabla \phi(t,y) \big) d \rho_k(t, x) \rho_k(t, y),
\end{align*}
pointwise in $t$,  and the latter converges to the desired term by applying again Proposition \ref{narroConvProd} to the product $\rho_k(t) \otimes \rho_k(t)$. 

\smallskip

Let us now assume that $V$ is $\lambda'$-convex for some $\lambda' \leq 0$ and consider a minimizing sequence $(\nu_k, \rho_k)$ for $\J + \chi_B$ in the set $A'$.  Once again, without loss of generality we can assume that $\J(\nu_k, \rho_k) + \chi_B(\nu_k, \rho_k) < \infty$ for every $k \in \N$, thus $\rho_k$ is a weak measure solution of \eqref{eq: TE gradient flow} with $\nu_k$ and initial datum $\rho_0$. 

We will see that for each $k$, $\rho_k$ is absolutely continuous with respect to the $2$-Wasserstein distance.  Indeed,  thanks to Theorem 8.3.1 \cite{AGS} and the fact that $\rho_k$ is a weak measure solution,  we only need to show that $\| v(t)  \|_{L^2(\rho_k(t), \R^d)} \in L^1(0,T)$,  where $v(t)$ is the velocity field of the continuity equation \eqref{eq: TE gradient flow}.  In our case we can estimate
\begin{align*}
\| v(t) \|^2_{L^2(\rho_k(t), \R^d)} & = \int_{\R^d} | v(t)(x) |^2 d\rho_k(t,x) \\
& \leq 2 \int_{\R^d} \left( \int_{x\neq y} |\nabla W(x-y)|^2 d\rho_k(t,y) + M \int_{\R^d} |\nabla V(x-y)|^2 d\nu_k(t,y) \right) d\rho_k(t,x) \\
& \leq 2 (Lip (W) + M Lip (V))^2 < \infty, 
\end{align*} 
thus providing the desired absolute continuity. 

We are then in position to apply Proposition \ref{uniqueness of weak solutions} and deduce that $\rho_k$ is the unique weak measure solution of \eqref{eq: TE gradient flow} with $\nu_k$ and initial datum $\rho_0$.  As a consequence,  it must coincide with the one provided by Theorem \ref{thm: TE gradient flows} in the space $\Lip_{L,d_{W_2}}(0,T;\P_2)$.  We then deduce that $(\rho_k)_k \subset \Lip_{L,d_{W_2}}(0,T;\P_2)$ and we conclude by the previous part of the proof.   
\end{proof}

\section{Conclusions and perspectives}
We studied existence of solutions for optimal control problems associated to nonlocal transport equations, used in the modelling of the behaviour of a population of individuals influenced by the presence of  control agents.  The results are proved for a class of mildly singular potentials in a gradient flow formulation for the target transport equation. A natural extension to the present paper will be disclosed dealing with essentially singular potentials as Coulomb or Lennard-Jones type ones. However, for the study of such potentials we expect to adopt different techniques from the one used above, since, up to the authors' knowledge, the hypotheses on the kernels in  \emph{\hyperlink{self target}{(Self)},  \hyperlink{cross target}{(Cross)}} are minimal in the Optimal Transport framework. Moreover, another interesting aspect lies in the numerical discretization of \eqref{control-equation}. Among the others we expect that this task could be performed on one hand using the combination of the two variational formulations (the J.K.O.-scheme and the optimisation problem for the cost functional), on the other hand through a  deterministic reconstructions of the densities starting from \eqref{eq:control_discrete}. We leave these topics for future works.
\section*{Acknowledgments}
SF acknowledges the support from the 04ATE2021-grant "Mathematical models for social innovations: vehicular and pedestrian traffic,  opinion formation and seismology" of the University of L'Aquila. AK  thanks Prof. Maria Colombo for the support of his internship in the AMCV group.

\end{document}